\documentclass[10pt,a4paper]{scrartcl}

\usepackage{amssymb}
\usepackage{amsmath}
\usepackage{amsfonts}
\usepackage{bbm}
\usepackage{amsthm}
\usepackage{mathrsfs}	
\usepackage[hidelinks]{hyperref}\usepackage{color}
\usepackage[margin=2.5cm]{geometry}
\usepackage[all,cmtip]{xy}
\usepackage[utf8]{inputenc}
\usepackage{graphicx}
\usepackage{varwidth}

\let\emph\undefined\newcommand{\emph}[1]{\textsl{#1}}
\usepackage{upgreek}
\usepackage{rotating}
\usepackage{tikz}
\usepackage{accents}
\usetikzlibrary{matrix,arrows,decorations.pathmorphing,shapes.geometric}
\usepackage{tikz-cd}
\usepackage{needspace}
\newcommand{\spaceplease}{\needspace{5\baselineskip}}

\usetikzlibrary{decorations.markings}
\usetikzlibrary{backgrounds}

\usepackage{etex}
\usetikzlibrary{svg.path}

\tikzstyle{tikzfig}=[baseline=-0.25em,scale=0.5]

\pgfkeys{/tikz/tikzit fill/.initial=0}
\pgfkeys{/tikz/tikzit draw/.initial=0}
\pgfkeys{/tikz/tikzit shape/.initial=0}
\pgfkeys{/tikz/tikzit category/.initial=0}

\pgfdeclarelayer{edgelayer}
\pgfdeclarelayer{nodelayer}
\pgfsetlayers{background,edgelayer,nodelayer,main}

\tikzstyle{none}=[inner sep=0mm]

\newcommand{\tikzfig}[1]{%
	{\tikzstyle{every picture}=[tikzfig]
		\IfFileExists{#1.tikz}
		{\input{#1.tikz}}
		{%
			\IfFileExists{./figures/#1.tikz}
			{\input{./figures/#1.tikz}}
			{\tikz[baseline=-0.5em]{\node[draw=red,font=\color{red},fill=red!10!white] {\textit{#1}};}}%
	}}%
}

\tikzstyle{every loop}=[]

\usepackage{tikzit}

\tikzstyle{black dot}=[fill=black, draw=black, shape=circle, minimum size=3pt, inner sep=0pt]
\tikzstyle{black dot small}=[fill=black, draw=black, shape=circle, minimum size=3pt, inner sep=0pt]
\tikzstyle{wbox}=[fill=white, draw=black, shape=rectangle, minimum height=0.5cm, minimum width=0.01cm]
\tikzstyle{bbox}=[fill=white, draw=black, shape=rectangle, minimum height=0.5cm, minimum width=0.01cm]
\tikzstyle{rbox}=[fill=white, draw=black, densely dotted, shape=rectangle, minimum height=0.5cm, minimum width=0.01cm]
\tikzstyle{bwbox}=[draw=black, shape=rectangle, minimum width=2cm, minimum height=0.5cm]
\tikzstyle{bbwbox}=[draw=black, shape=rectangle, minimum width=1cm, minimum height=1cm]
\tikzstyle{big white circle}=[fill=white, draw=black, shape=circle, minimum width=0.75cm]
\tikzstyle{white dot big}=[fill=white, draw=black, shape=circle, inner sep=1pt]
\tikzstyle{white dot}=[fill=white, draw=black, shape=circle, minimum size=3pt, inner sep=0pt]
\tikzstyle{flat box}=[fill=white, draw=black, shape=rectangle, minimum width=1.3cm, minimum height=0.5cm]
\tikzstyle{square}=[fill=white, draw=black, shape=rectangle]
\tikzstyle{flat box 2}=[fill=white, draw=black, shape=rectangle, minimum height=0.5cm, minimum width=0.01cm]
\tikzstyle{bigbox}=[fill=white, draw=black, shape=rectangle, minimum height=0.5cm, minimum width=0.8cm]
\tikzstyle{over }=[front]
\tikzstyle{bigdisk}=[draw=black, shape=circle, minimum width=3cm]
\tikzstyle{wdisk}=[shape=circle, minimum width=0.48cm,fill=white]
\tikzstyle{bigdisk2}=[draw=black, fill=lightgray, shape=circle, minimum width=3cm]
\tikzstyle{little disk}=[fill=white, draw=black, shape=circle, minimum width=0.5cm]

\tikzstyle{mid arrow}=[-, postaction={on each segment={mid arrow}}]
\tikzstyle{end arrow}=[->]
\tikzstyle{mover}=[-, link]
\tikzstyle{mydots}=[-,dotted]
\tikzstyle{thick}=[-,line width=1pt]
\tikzstyle{dotarrow}=[->,dotted,draw=black,line width=1pt]
\tikzstyle{bdotarrow}=[->,dotted,draw=black,line width=1pt]
\tikzstyle{red mid arrow}=[-, draw={rgb,255: black,214; green,42; black,51}, postaction={on each segment={mid arrow}}, line width=1pt]
\tikzstyle{RED}=[-, draw={rgb,255: red,214; green,42; black,51}]
\tikzstyle{REDdotted}=[-,dashed]
\tikzstyle{reddots}=[-,dotted, draw={rgb,255: red,214; green,42; black,51}]
\tikzstyle{bluedashed}=[-,dashed, draw=black]
\tikzstyle{blue}=[-, draw=black]
\tikzstyle{blue mid arrow}=[-, draw={rgb,255: red,23; green,37; black,167}, postaction={on each segment={mid arrow}}, line width=1pt]
\tikzstyle{over}=[-, link]
\tikzstyle{mover}=[-, link]
\tikzstyle{mapsto}=[{|->}]

\usetikzlibrary{decorations.pathreplacing,decorations.markings}
\tikzset{
	on each segment/.style={
		decorate,
		decoration={
			show path construction,
			moveto code={},
			lineto code={
				\path [#1]
				(\tikzinputsegmentfirst) -- (\tikzinputsegmentlast);
			},
			curveto code={
				\path [#1] (\tikzinputsegmentfirst)
				.. controls
				(\tikzinputsegmentsupporta) and (\tikzinputsegmentsupportb)
				..
				(\tikzinputsegmentlast);
			},
			closepath code={
				\path [#1]
				(\tikzinputsegmentfirst) -- (\tikzinputsegmentlast);
			},
		},
	},
	mid arrow/.style={postaction={decorate,decoration={
				markings,
				mark=at position .7 with {\arrow[#1]{stealth}}
	}}},
}
\tikzset{%
	link/.style    = { white, double = black, line width = 1.8pt,
		double distance = 0.4pt },
	channel/.style = { white, double = black, line width = 0.8pt,
		double distance = 0.8pt },
}

\usepackage{mathtools}
\mathtoolsset{showonlyrefs}
\usepackage{epstopdf}
\usepackage{pdfpages}

\newtheoremstyle{mytheorem}
{\topsep}
{\topsep}
{\slshape}
{0pt}
{\bfseries}
{.}
{ }
{\thmname{#1}\thmnumber{ #2}\thmnote{ {\normalfont\slshape(#3)}}}

\newtheoremstyle{mydefinition}
{\topsep}
{\topsep}
{\normalfont}
{0pt}
{\bfseries}
{.}
{ }
{\thmname{#1}\thmnumber{ #2}\thmnote{ {\normalfont\slshape(#3)}}}

\theoremstyle{mytheorem}
\newtheorem{theorem}{Theorem}[section]

\makeatletter
\newtheorem*{rep@theorem}{\rep@title}
\newcommand{\newreptheorem}[2]{%
	\newenvironment{rep#1}[1]{%
		\def\rep@title{#2 \ref{##1}}%
		\begin{rep@theorem}}%
		{\end{rep@theorem}}}
\makeatother

\newreptheorem{theorem}{Theorem}
\newreptheorem{corollary}{Corollary}
\newtheorem{lemma}[theorem]{Lemma}
\newtheorem{proposition}[theorem]{Proposition}
\newtheorem{corollary}[theorem]{Corollary}
\theoremstyle{mydefinition}

\newtheorem{exx}[theorem]{Example}
\newenvironment{example}
{\pushQED{\qed}\exx}
{\popQED\endexx}
\newtheorem{remm}[theorem]{Remark}
\newenvironment{remark}
{\pushQED{\qed}\remm}
{\popQED\endremm}
\numberwithin{equation}{section}
\usepackage{enumitem}

\newenvironment{pnum}{\begin{enumerate}[topsep=2pt,parsep=2pt,partopsep=2pt,itemsep=0pt,label={(\roman{*})}]}{\end{enumerate}}

\DeclareMathSymbol{\Phiit}{\mathalpha}{letters}{"08}\let\Phi\undefined\newcommand{\Phi}{\Phiit}
\DeclareMathSymbol{\Psiit}{\mathalpha}{letters}{"09}\let\Psi\undefined\newcommand{\Psi}{\Psiit}
\DeclareMathSymbol{\Sigmait}{\mathalpha}{letters}{"06}\let\Sigma\undefined\newcommand{\Sigma}{\Sigmait}
\DeclareMathSymbol{\Xiit}{\mathalpha}{letters}{"04}
\DeclareMathSymbol{\Lambdait}{\mathalpha}{letters}{"03}\let\Lambda\undefined\newcommand{\Lambda}{\Lambdait}
\DeclareMathSymbol{\Piit}{\mathalpha}{letters}{"05}\let\Pi\undefined\newcommand{\Pi}{\Piit}
\DeclareMathSymbol{\Gammait}{\mathalpha}{letters}{"00}\let\Gamma\undefined\newcommand{\Gamma}{\Gammait}
\DeclareMathSymbol{\Omegait}{\mathalpha}{letters}{"0A}\let\Omega\undefined\newcommand{\Omega}{\Omegait}
\DeclareMathSymbol{\Upsilonit}{\mathalpha}{letters}{"07}\let\Upsilon\undefined\newcommand{\Upsilon}{\Upilonit}
\DeclareMathSymbol{\Thetait}{\mathalpha}{letters}{"02}\let\Theta\undefined\newcommand{\Theta}{\Thetait}

\def\Hom{\mathrm{Hom}}

\def\End{\mathrm{End}}
\def\id{\mathrm{id}}

\let\to\undefined\newcommand{\to}{\longrightarrow}
\let\mapsto\undefined\newcommand{\mapsto}{\longmapsto}
\newcommand{\catf}[1]{\mathsf{#1}}
\newcommand{\Proj}{\operatorname{\catf{Proj}}}

\def\op{\mathrm{op}}

\newcommand{\dif}{\text{d}}

\newcommand{\IN}{{\normalfont\tiny in}}
\newcommand{\OUT}{{\normalfont\tiny out}}
\usepackage[normalem]{ulem}

\newcommand{\ra}[1]{\ \xrightarrow{\ \  #1  \  \ }\ }

\def\Ch{\catf{Ch}_k}

\newcommand{\cat}[1]{\mathcal{#1}}

\newcommand{\gbos}{{\gamma^\bullet_*}^\op}

\newcommand{\Ext}{\operatorname{Ext}}
\newcommand{\flint}{\int_{\text{\normalfont f}\mathbb{L}}}
\newcommand{\alg}{\catf{A}}\newcommand{\coalg}{\catf{F}}\newcommand{\somealg}{\catf{T}}\newcommand{\isomealg}{{\catf{T}^\bullet}}\newcommand{\somecoalg}{\catf{K}}
\newcommand{\ialg}{{\catf{A}^\bullet}}

\let\mod\undefined
\newcommand{\mod}{\catf{-mod}}

\newcommand{\rint}{\int^\mathbb{R}}

\newcommand{\frint}{\int^{\text{\normalfont f}\mathbb{R}}}

\newcommand{\bspace}[2]{\left( \  #1 \ , \    #2  \ \right)}
\makeatletter
\newcommand{\ostar}{\mathbin{\mathpalette\make@circled\star}}
\newcommand{\make@circled}[2]{%
	\ooalign{$\m@th#1\smallbigcirc{#1}$\cr\hidewidth$\m@th#1#2$\hidewidth\cr}%
}
\newcommand{\smallbigcirc}[1]{%
	\vcenter{\hbox{\scalebox{0.77778}{$\m@th#1\bigcirc$}}}%
}
\makeatother
\newcommand{\mlabel}[1]{{\footnotesize $#1$}}

\setkomafont{caption}{\small\slshape}

\setkomafont{section}{\rmfamily\large}
\setkomafont{subsection}{\rmfamily}
\setkomafont{subsubsection}{\rmfamily}
\setkomafont{subparagraph}{\normalfont\scshape}

\newtheorem*{theorem*}{Theorem}
\newtheorem*{corollary*}{Corollary}

\makeatletter
\renewcommand\section{\@startsection {section}{1}{\z@}%
	{-3.5ex \@plus -1ex \@minus -.2ex}%
	{2.3ex \@plus.2ex}%
	{\normalfont\scshape\centering}}
\makeatother

\usepackage{titletoc}
\dottedcontents{section}[1em]{\rmfamily}{1em}{0.2cm}
\dottedcontents{subsection}[0em]{}{3.3em}{1pc}

\begin{document} 
\vspace*{-10mm}
	\begin{flushright}
		\small
		{\sffamily [ZMP-HH/22-8]} \\
		\textsf{Hamburger Beiträge zur Mathematik Nr.~920}\\
		\textsf{CPH-GEOTOP-DNRF151}
	\end{flushright}
	
\vspace{8mm}
	
	\begin{center}
		\textbf{\large{Homotopy Invariants of Braided Commutative Algebras \\[0.5ex] and the Deligne Conjecture for Finite Tensor Categories}}\\
		\vspace{6mm}

{\large Christoph Schweigert ${}^a$ and  Lukas Woike ${}^b$}

\vspace{3mm}

\normalsize
{\slshape $^a$ Fachbereich Mathematik\\ Universit\"at Hamburg\\
	Bereich Algebra und Zahlentheorie\\
	Bundesstra\ss e 55\\  D-20146 Hamburg }

\vspace*{3mm}	

{\slshape $^b$ Institut for Matematiske Fag\\ K\o benhavns Universitet\\
	Universitetsparken 5 \\  DK-2100 K\o benhavn \O }

	\end{center}
\vspace*{1mm}
	\begin{abstract}\noindent 
		It is easy to find algebras $\mathbb{T}\in\mathcal{C}$ in a finite tensor category $\mathcal{C}$ that naturally come with a lift to a braided commutative algebra $\mathsf{T}\in Z(\mathcal{C})$ in the Drinfeld center of $\mathcal{C}$. In fact, any finite tensor category has at least two such algebras, namely the monoidal unit $I$ and the canonical end $\int_{X\in\mathcal{C}} X\otimes X^\vee$. Using the theory of braided operads, we prove that for any such algebra $\mathbb{T}$ the homotopy invariants, i.e.\ the derived morphism space from $I$ to $\mathbb{T}$, naturally come with the structure of a differential graded $E_2$-algebra. This way, we obtain a rich source of differential graded $E_2$-algebras in the homological algebra of finite tensor categories. We use this result to prove that Deligne's $E_2$-structure on the Hochschild cochain complex of a finite tensor category is induced by the canonical end, its multiplication and its non-crossing half braiding. With this new and more explicit description of Deligne's $E_2$-structure, we can lift the Farinati-Solotar bracket on the Ext algebra of a finite tensor category to an $E_2$-structure at cochain level. Moreover, we prove that, for a unimodular pivotal finite tensor category, the inclusion of the Ext algebra into the Hochschild cochains is a monomorphism of framed $E_2$-algebras, thereby refining a result of Menichi.
\end{abstract}

\tableofcontents

\normalsize

\section{Introduction and summary}
Over the last two decades,
the notion of a \emph{finite tensor category} developed 
in \cite{etingofostrik}
and then later presented comprehensively in the monograph \cite{egno}
has become one of the standard frameworks in quantum algebra:
A finite tensor category over a field $k$ (that we will assume to be algebraically closed throughout) is a $k$-linear abelian rigid monoidal category 
with simple unit, finite-dimensional morphism spaces, enough projective objects, finitely many simple objects up to isomorphism subject to the requirement that every object has finite length.
The notion has proven to be
sufficiently restrictive to prove a great number of very strong results.
At the same time, it remains flexible enough
to allow for an inexhaustible supply of examples coming from Hopf algebras or vertex operator algebras, see e.g.\
\cite{kassel,huang,hlz}.  
Moreover, finite tensor categories are intimately related to low-dimensional topology, more precisely to topological field theories and modular functors
\cite{rt2,turaev,kl,BDSPV15}
 --- both in the semisimple case and beyond semisimplicity.

In the non-semisimple situation,
the homological algebra of finite tensor categories is a rich and well-studied subject
	\cite{gk,etingofostrik,farinatisolotar,mpsw,menichi,bichon,hermann,lq,negronplavnik}.
The present article is concerned with the systematic construction 
and investigation of \emph{higher multiplicative structures} in the homological algebra of finite tensor categories, more specifically \emph{differential graded $E_2$-algebras}. 
The methods will be used to prove a number of results on the `standard' homological algebra quantities of a finite tensor category (possibly with more structure), namely the Hochschild cochains and the Ext algebra, but will also be used for the construction of new multiplicative structures.

Before giving the precise statements and our motivation, let us recall that the $E_2$-operad is the topological operad whose space $E_2(n)$ of $n$-ary operations is the space of embeddings $(\mathbb{D}^2)^{\sqcup n}\to\mathbb{D}^2$ built from translations and rescalings; we refer to \cite{may,Fresse1} for a textbook treatment. A differential graded $E_2$-algebra $A$ is  a (co)chain complex $A$ over $k$ together with maps $C_*(E_2(n);k)\to [A^{\otimes n},A]$, where $C_*(-;k)$ is the functor taking $k$-chains and $[-,-]$ denotes the	 hom complex. These maps are subject to the usual requirements regarding equivariance and operadic composition.
Since $E_2(n)$ is a classifying space of the (pure) braid group on $n$ strands, an $E_2$-algebra is an algebra whose commutativity behavior is controlled in a coherent way through the braid group. The (co)homology of an $E_2$-algebra is a Gerstenhaber algebra \cite{cohen}. If, in addition to translations and rescalings of disks, we allow rotations, we obtain the \emph{framed $E_2$-operad}. The (co)homology of a differential graded framed $E_2$-algebra is a Batalin-Vilkovisky algebra \cite{getzler}. 
It is impractical to construct differential graded $E_2$-algebras by exhibiting the needed action maps $C_*(E_2(n);k)\to [A^{\otimes n},A]$ by hand, and in fact, a large part of this article is devoted to developing efficient methods for the construction of $E_2$-algebras that are tailored towards applications in the homological algebra of finite tensor categories.

The article is centered around the following very natural problem:
In any finite tensor category $\cat{C}$,
one can define the canonical end $\mathbb{A}=\int_{X\in\cat{C}} X\otimes X^\vee$ (here $X^\vee$ denotes the dual of $X$), and it is a standard observation that the maps $X\otimes X^\vee \otimes X\otimes X^\vee\to X\otimes X^\vee$ contracting the middle two tensor factors via an evaluation endow $\mathbb{A}$ with the structure of an algebra in $\cat{C}$.  As a consequence,
the space $\cat{C}(I,\mathbb{A})$ of morphisms
from the monoidal unit $I$
to $\mathbb{A}$ becomes an algebra as well --- this time in vector spaces --- and it is well-known that this algebra is commutative. 
The vector space $\cat{C}(I,\mathbb{A})$ can be canonically identified with the zeroth Hochschild cohomology 
$HH^0(\cat{C})=\int_{X\in\cat{C}} \cat{C}(X,X)$
of $\cat{C}$ (seen here just as a linear category). The vector space $HH^0(\cat{C})$ inherits a commutative product
from the composition of morphisms
(both for the definition of $\mathbb{A}$ and $HH^0(\cat{C})$, we can restrict the ends to the subcategory $\Proj\cat{C}\subset \cat{C}$ of projective or, equivalently, injective objects without changing the value of the end).
With these products, the isomorphism
$\cat{C}(I,\mathbb{A})\cong HH^0(\cat{C})$
is an isomorphism of algebras.
The isomorphism $\cat{C}(I,\mathbb{A})\cong HH^0(\cat{C})$
of vector spaces
 has a natural generalization:
To this end, consider the homotopy invariants of $\mathbb{A}$, i.e.\ the derived morphism space $\cat{C}(I,\mathbb{A}^\bullet)$, where $\mathbb{A}^\bullet$ is an injective resolution of $\mathbb{A}$; its cohomology is $\Ext^*_\cat{C}(I,\mathbb{A})$.
  The cochain complex $\cat{C}(I,\mathbb{A}^\bullet)$
  		 is canonically equivalent to the Hochschild cochains of $\cat{C}$, 
i.e.\ the homotopy end $\rint_{X\in\Proj \cat{C}}    \cat{C}(X,X) $ of the endomorphism spaces of projective objects; 
\begin{align}
	\cat{C}(I,\mathbb{A}^\bullet) \simeq \rint_{X\in\Proj\cat{C}} \cat{C}(X,X) \ , \label{eqnAHH}
\end{align}
see e.g.\ \cite[Section~2.2]{bichon} for this statement in Hopf-algebraic form which goes back to \cite{cartaneilenberg}.
 The Hochschild cochain complex
$\rint_{X\in\Proj\cat{C}} \cat{C}(X,X)$, 
as the Hochschild cochain complex of \emph{any} linear category,
comes with more structure. By a result of Gerstenhaber \cite{gerstenhaber},
its cohomology is a Gerstenhaber algebra. Deligne famously conjectured 
in 1993 that this Gerstenhaber structure comes in fact from an $E_2$-structure 
at cochain level. This conjecture was proven in a 
variety of different ways \cite{tamarkin,cluresmith,bergerfresse}. 
For symmetric Frobenius algebras, the $E_2$-structure 
even extends to the structure of an algebra 
over the framed $E_2$-operad
\cite{tradlerzeinalian,costellotcft,kaufmann}.
Given that in zeroth cohomology, \eqref{eqnAHH}
reduces to the 
 isomorphism $\cat{C}(I,\mathbb{A})\cong HH^0(\cat{C})$ of commutative algebras, the following question is evident:
 \begin{itemize}
 	\item[(Q)] \slshape
Is there a structure on $\mathbb{A}$ or rather on an injective resolution $\mathbb{A}^\bullet$ that turns the homotopy invariants $\cat{C}(I,\mathbb{A}^\bullet)$ into an $E_2$-algebra that solves the Deligne Conjecture, thereby turning \eqref{eqnAHH} into an equivalence of $E_2$-algebras? 
If so, what is the needed structure?	
Or in a more pointed, but less precise way: Can  one use the canonical end
$\mathbb{A}$ to give a solution to Deligne's Conjecture?\normalfont\label{questionQ}
\end{itemize}

In order to answer this question, it will be beneficial to study multiplicative structures on homotopy invariants more generally:
Let $\mathbb{T}$ be an algebra inside a finite tensor category $\cat{C}$
that lifts to a 
braided commutative algebra $\somealg$ in the Drinfeld center $Z(\cat{C})$, i.e.\
$\mathbb{T}=U\somealg$ as algebras in $\cat{C}$, where $U:Z(\cat{C})\to\cat{C}$ is the forgetful functor,
and $\mu_\somealg \circ c_{\somealg,\somealg}=\mu_\somealg$ for the multiplication $\mu_\somealg$ of $\somealg$ and the braiding $c_{\somealg,\somealg}$ in $Z(\cat{C})$. 
Using the theory of braided operads, 
we prove the following 
very general result 
on the multiplicative structure
on homotopy invariants
$\cat{C}(I,\mathbb{T}^\bullet)$: 

\begin{reptheorem}{thmE2derivedhom}
	Let $\mathbb{T} \in \cat{C}$ be an algebra in a finite tensor category $\cat{C}$
together with a lift to a braided commutative algebra $\somealg \in Z(\cat{C})$ in the Drinfeld center.
Then the multiplication of $\mathbb{T}$ and the half braiding of $\mathbb{T}$ induce the structure of an $E_2$-algebra on the space $\cat{C}(I,\mathbb{T}^\bullet)$ of homotopy invariants of $\mathbb{T}$. 
\end{reptheorem}

In particular, $\Ext_\cat{C}^*(I,\mathbb{T})$ becomes a Gerstenhaber algebra.
Of course, Theorem~\ref{thmE2derivedhom} includes cases in which 
$\cat{C}(I,\mathbb{T}^\bullet)$ will be even `more commutative' 
(obviously, we find vector space valued commutative algebras in the semisimple case), 
but the result includes examples with a non-trivial Gerstenhaber bracket. In this sense, the `2' in $E_2$ is sharp
(Remark~\ref{remsharp}).

The connection between
Theorem~\ref{thmE2derivedhom}
and question~(Q) is as follows:
The right adjoint $R: \cat{C}\to Z(\cat{C})$ to the forgetful functor $U:Z(\cat{C})\to \cat{C}$ sends the unit $I$ of $\cat{C}$ to a braided commutative algebra $\alg=R(I) \in Z(\cat{C})$  thanks to a result of Davydov, Müger, Nikshych and Ostrik \cite{dmno}.
The underlying  object $\mathbb{A}=UR(I)=\int_{X\in\cat{C}}X\otimes X^\vee$ is the canonical end of the finite tensor category $\cat{C}$; and in fact, $U\alg =\mathbb{A}$ as algebras. By Theorem~\ref{thmE2derivedhom} the homotopy invariants $\cat{C}(I,\mathbb{A}^\bullet)$ now inherit a multiplication from the multiplication of $\mathbb{A}$ and its half braiding (that is called the \emph{non-crossing half braiding}). This specific $E_2$-structure is a solution to Deligne's Conjecture and hence an answer to question~(Q):

 \begin{reptheorem}{thhmcomparisondeligne}[Comparison Theorem]
 	For any finite tensor category $\cat{C}$, 
 	the algebra structure on the canonical end $\mathbb{A}=\int_{X\in\cat{C}} X\otimes X^\vee$ and its canonical lift to the Drinfeld center induces  an $E_2$-algebra structure on the 
 	homotopy invariants $\cat{C}(I,\mathbb{A}^\bullet)$. Under the equivalence $\cat{C}(I,\mathbb{A}^\bullet)\simeq \rint_{X\in \Proj \cat{C}} \cat{C}(X,X)$, this $E_2$-structure
 	provides a solution to Deligne's Conjecture in the sense that it induces the standard Gerstenhaber structure on the Hochschild cohomology of $\cat{C}$.
 \end{reptheorem}
 
 Since the definition of $\mathbb{A}=\int_{X\in\cat{C}} X\otimes X^\vee$ makes use of the monoidal structure of $\cat{C}$
 while the standard Gerstenhaber structure on the Hochschild cohomology of $\cat{C}$ sees only the underlying linear structure,
 the comparison result of Theorem~\ref{thhmcomparisondeligne}
 is very non-obvious.
 As a result, the proof of Theorem~\ref{thhmcomparisondeligne} is unfortunately very involved and occupies a large portion of the article.

 \subparagraph{Our motivation.}
Theorem~\ref{thhmcomparisondeligne} provides a new 
 proof of Deligne's Conjecture
for finite tensor categories, but giving a new proof of the conjecture is
not our main motivation.  Instead, we are motivated by the fact that the new description of Deligne's $E_2$-structure on the Hochschild cochains of a finite tensor category in terms of the canonical end is significantly simpler. 
This description is one of the cornerstones of the proof of the differential graded Verlinde formula in \cite{vd} for the differential graded modular functor of a modular category \cite{dmf}.

 \subparagraph{Other applications.}
Moreover, 
Theorem~\ref{thmE2derivedhom} can be used to construct \emph{other} $E_2$-algebras. We prove for example:

\begin{repcorollary}{corbraided}
	Let $\cat{C}$ be a braided finite tensor category.
	For any braided commutative algebra $B\in\bar{\cat{C}}\boxtimes\cat{C}$, denote by $B_\otimes$ the algebra in $\cat{C}$ obtained by applying the monoidal product functor to $B$. 
	Then the homotopy invariants $\cat{C}(I,B_\otimes^\bullet)$ of $B_\otimes$ naturally form an $E_2$-algebra. 
\end{repcorollary}

As a special case, this contains the \emph{dolphin algebra} needed as a critical auxiliary object in \cite{vd}, see Example~\ref{exampledolphin}.

If we apply Theorem~\ref{thmE2derivedhom} to
the monoidal unit, we obtain an $E_2$-structure on the algebra $\cat{C}(I,I^\bullet)$ of self-extensions of the monoidal unit, also called the \emph{Ext algebra}.
It induces on cohomology a well-known graded commutative product \cite{etingofostrik}. Its Gerstenhaber bracket is a generalization
of the Farinati-Solotar bracket \cite{farinatisolotar}
 to arbitrary finite tensor categories.
 
 \begin{repcorollary}{corcohomftc}
 	Let $\cat{C}$ be a finite tensor category. 
 	The self-extension algebra
 	$\cat{C}(I,I^\bullet)$  carries the structure of an $E_2$-algebra
 	that after taking cohomology induces the Farinati-Solotar Gerstenhaber bracket.
 	With this $E_2$-structure, there 
 	is a canonical map
 	$
 		\cat{C}(I,I^\bullet)\to \rint_{X\in \Proj\cat{C}}\cat{C}(X,X)  
 	$ to the Hochschild cochain complex of $\cat{C}$ equipped with the usual $E_2$-structure. This map is a map of $E_2$-algebras. After taking cohomology, it induces a monomorphism $
 		\Ext_\cat{C}^*(I,I) \to HH^*(\cat{C}) $
 	of Gerstenhaber algebras (with suitable models, it is also a monomorphism at cochain level).
 \end{repcorollary}
 
 One of the key advantages of Theorem~\ref{thmE2derivedhom} is the possibility to find, with relatively little effort, extensions to \emph{framed} $E_2$-algebras. For unimodular pivotal finite tensor categories, we can give the following framed extension of Corollary~\ref{corcohomftc}:
\begin{repcorollary}{corgenmen}
	For any unimodular pivotal finite tensor category $\cat{C}$, both the self-extension algebra $\cat{C}(I,I^\bullet)$ and the Hochschild cochain complex $\rint_{X\in\Proj\cat{C}}\cat{C}(X,X)$ come equipped with  a framed $E_2$-algebra structure such that the map
	$	\cat{C}(I,I^\bullet)\to \rint_{X\in \Proj\cat{C}}\cat{C}(X,X) $  
 is a map (and with suitable models even a monomorphism) of framed $E_2$-algebras. In cohomology, it induces a monomorphism
	$
		\Ext_\cat{C}^*(I,I) \to HH^*(\cat{C}) 
	$ of Batalin-Vilkovisky algebras.
\end{repcorollary}
This generalizes a result of Menichi~\cite{menichi} who previously proved the result at cohomology level for unimodular pivotal Hopf algebras by giving a Batalin-Vilkovisky structure
(i.e.\ the structure of an algebra over the homology of the framed $E_2$-operad)
on the self-extension algebra.

\subparagraph{Conventions.}
 \begin{enumerate}[topsep=2pt,parsep=2pt,partopsep=2pt,itemsep=0pt,label={(\arabic{*})}]
	\item
	We fix for the entire article an algebraically closed field $k$.

	\item We denote by
 $\Ch$  
the symmetric monoidal category of chain complexes over $k$.
Unless otherwise stated, it will be equipped with its projective model structure. In this model structure, the weak equivalences (for short: equivalences) are quasi-isomorphisms while the fibrations are degree-wise surjections.
We call
a  (co)fibration which is also an equivalence a \emph{trivial (co)fibration}.
Equivalences will be denoted by $\simeq$ (we reserve $\cong$ for isomorphisms).
As a \emph{(canonical) equivalence} between chain complexes, we understand a (canonical) zigzag of equivalences.
We assume 
 functors between linear and differential graded categories 
 automatically to be enriched.

\item
Concerning the convention for duality in monoidal categories, we follow~\cite[Section~2.10]{egno}:
Let $\cat{C}$ be a rigid monoidal category and $X\in\cat{C}$.  We denote \begin{itemize}
	\item the  \emph{left dual}  by $X^\vee$
and by $d_X:X^\vee \otimes X\to I$ and 
$b_X:I\to X\otimes X^\vee$
the 
 evaluation 
 and 
 coevaluation, respectively,

\item 
the  \emph{right dual} by ${^\vee \! X}$ 
and by $\widetilde d_X : X\otimes  {^\vee \! X} \to I$ and 
$\widetilde b_X : I \to {^\vee \! X}\otimes X$
the 
evaluation 
and 
coevaluation, respectively.
\end{itemize}
Left and right duality induce the following adjunction isomorphisms for $X,Y,Z\in\cat{C}$:
\begin{align}
\cat{C}(X\otimes Y,Z)&\cong \cat{C}(X,Z\otimes Y^\vee) \ , \\
\cat{C}(Y^\vee\otimes X,Z)&\cong \cat{C}(X,Y\otimes Z) \ , \\
\cat{C}(X\otimes {^\vee \! Y}, Z) &\cong \cat{C}(X,Z\otimes Y)\ ,\\
\cat{C}(Y\otimes X,Z)&\cong \cat{C}(X,{^\vee \! Y} \otimes Z)  
\end{align}

\item 
If $\cat{C}$ is a finite (tensor) category, then it is a module category over the symmetric monoidal category of finite-dimensional vector spaces over $k$.
As a result, we can form the tensoring $V\otimes X\in \cat{C}$ of a finite-dimensional vector space $V$ with an object $X\in \cat{C}$ and similarly a powering $X^V=V^*\otimes X\in \cat{C}$, where $V^*$ is the dual of $V$.

\item
In several places, we will use the \emph{graphical calculus} for morphisms in (braided) monoidal categories, see e.g.\ \cite{kassel}. 
In this calculus, objects are vertical lines and the monoidal product corresponds to the  juxtaposition of lines (the monoidal unit is the empty collection of lines) while composition is represented by vertical stacking. We denote the braiding and inverse braiding as an overcrossing and undercrossing, respectively, 
 while  evaluation and coevaluation are written as a cap and cup, respectively.
 Morphisms should be read from bottom to top. \label{congraphcalc}
 
 \end{enumerate}

 \subparagraph{Acknowledgments.} 
We would like to thank 
Adrien Brochier, 
Lukas Müller,
 Nathalie Wahl
 and Yang Yang
for helpful discussions. 

CS is supported by the Deutsche Forschungsgemeinschaft (DFG, German Research
Foundation) 
under Germany’s Excellence Strategy -- EXC 2121 ``Quantum Universe'' -- 390833306.
LW gratefully acknowledges support by 
the Danish National Research Foundation through the Copenhagen Centre for Geometry
and Topology (DNRF151)
and by the European Research Council (ERC) under the European Union's Horizon 2020 research and innovation programme (grant agreement No~772960). 

\noindent\begin{minipage}{0.85\textwidth}\raggedright
	This project has received funding from the European Union’s Horizon 2020 research and innovation
	programme under the Marie Sk\l odowska-Curie grant agreement No~101022691. \end{minipage}\hfill
\begin{minipage}{0.1\textwidth} \includegraphics[width=0.8\textwidth]{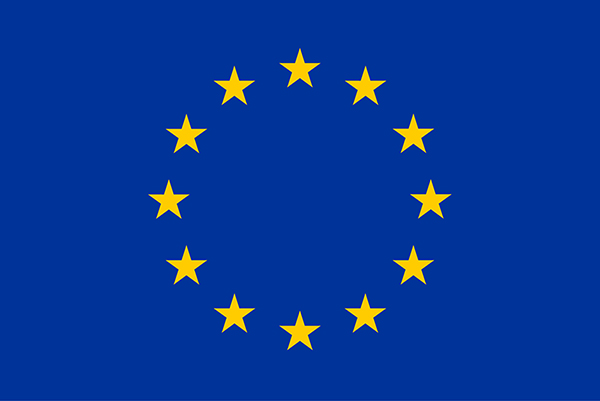}\end{minipage}

\section{A reminder on the Drinfeld center and related structures}
We assume that the reader is familiar with the notion of a \emph{braided finite tensor category}, i.e.\  a finite tensor category $\cat{B}$
that comes with natural isomorphisms $c_{X,Y}:X\otimes Y\cong Y\otimes X$ 	 subject to the usual hexagon axioms.
To any finite tensor category $\cat{C}$, we can assign a braided finite tensor category, namely its \emph{Drinfeld center}.
The Drinfeld center $Z(\cat{C})$ is the category whose objects are pairs of 
an object $X\in \cat{C}$ together with a \emph{half braiding} which is a natural isomorphism
natural isomorphism $X\otimes-\cong -\otimes X$ satisfying certain coherence conditions. Details on these notions can be found in~\cite[Chapters~6-8]{egno}.
By 
\cite[Theorem~3.34]{etingofostrik}
  $Z(\cat{C})$ is a finite tensor category again, and it is (by construction) braided (see also \cite[Theorem~3.8]{shimizuunimodular}). 

The forgetful functor $U:Z(\cat{C})\to\cat{C}$ is monoidal and exact and, for this reason, has both an oplax monoidal left adjoint $L:\cat{C}\to Z(\cat{C})$
and also a lax monoidal right adjoint $R: \cat{C}\to Z(\cat{C})$; we refer to \cite{bruguieresvirelizier,shimizuunimodular} for a detailed overview.
This implies the following: Since $I\in\cat{C}$ is naturally an algebra, \begin{align} \coalg:=LI\ , \quad \alg:=RI \end{align} 
inherit the structure of a coalgebra and an algebra in $Z(\cat{C})$, respectively.  
One can prove that the underlying objects of $\coalg$ and $\alg$ are given by the canonical coend and end, respectively, i.e.\
$U\coalg =\mathbb{F}=\int^{X\in\cat{C}} X^\vee \otimes X$
and $U\alg=\mathbb{A}=\int_{X\in\cat{C}} X\otimes X^\vee$. 
In order to give the half braiding $c_{\mathbb{F},Y}:\mathbb{F} \otimes Y \to Y\otimes \mathbb{F}$ and $c_{\mathbb{A},Y}:\mathbb{A} \otimes Y \to Y\otimes \mathbb{A}$ with $Y\in\cat{C}$ (that is often referred to as \emph{non-crossing half braiding}), it suffices  by the universal property of the (co)end to give the restriction to $X^\vee \otimes X \otimes Y$ and the component $X\otimes X^\vee \otimes Y$, respectively,
by
\begin{align} X^\vee \otimes X \otimes Y \ra{  (c_{\mathbb{F},Y})_X    } Y\otimes Y^\vee \otimes X^\vee \otimes X \otimes Y \cong Y \otimes \left(   (X\otimes Y)^\vee \otimes X\otimes Y     \right) \to Y \otimes \mathbb{F}    \label{firsthalfbraiding}  \\
	\mathbb{A}\otimes Y \to  \left(  Y\otimes X \otimes (Y\otimes X)^\vee \right)  \otimes Y\cong Y\otimes X\otimes X^\vee \otimes Y^\vee \otimes Y \ra{  (c_{\mathbb{A},Y})^X  } Y \otimes X \otimes X^\vee \ , \label{firsthalfbraiding2}
\end{align}
where the last map in \eqref{firsthalfbraiding} and the first map in \eqref{firsthalfbraiding2} are the structure maps of the coend and the end and where

\begin{align}
	{\footnotesize\begin{tikzpicture}[scale=0.65]
	\begin{pgfonlayer}{nodelayer}
		\node [style=none] (20) at (-11, 13) {$X$};
		\node [style=none] (23) at (-13, 13) {$X^\vee$};
		\node [style=none] (68) at (-18, 15.25) {$=$};
		\node [style=none] (70) at (-17, 17) {};
		\node [style=none] (71) at (-15, 17) {};
		\node [style=none] (72) at (-13, 17) {};
		\node [style=none] (73) at (-11, 17) {};
		\node [style=none] (74) at (-9, 17) {};
		\node [style=none] (75) at (-13, 14) {};
		\node [style=none] (76) at (-11, 14) {};
		\node [style=none] (77) at (-9, 14) {};
		\node [style=none] (78) at (-9, 13) {$Y$};
		\node [style=none] (79) at (-17, 18) {$Y$};
		\node [style=none] (80) at (-15, 18) {$Y^\vee$};
		\node [style=none] (82) at (-21, 15.25) {$(c_{\mathbb{F},Y})_X$};
		\node [style=none] (83) at (3, 13) {$X$};
		\node [style=none] (84) at (1, 13) {$Y$};
		\node [style=none] (85) at (-1, 15.25) {$=$};
		\node [style=none] (86) at (9, 14) {};
		\node [style=none] (87) at (7, 14) {};
		\node [style=none] (88) at (1, 17) {};
		\node [style=none] (89) at (3, 17) {};
		\node [style=none] (90) at (5, 17) {};
		\node [style=none] (91) at (1, 14) {};
		\node [style=none] (92) at (3, 14) {};
		\node [style=none] (93) at (5, 14) {};
		\node [style=none] (95) at (9, 13) {$Y$};
		\node [style=none] (96) at (7, 13) {$Y^\vee$};
		\node [style=none] (97) at (-4, 15.25) {$(c_{\mathbb{A},Y})^X$};
		\node [style=none] (98) at (5, 13) {$X^\vee$};
		\node [style=none] (99) at (-8, 15) {$ , $};
	\end{pgfonlayer}
	\begin{pgfonlayer}{edgelayer}
		\draw [bend right=90, looseness=1.75] (70.center) to (71.center);
		\draw (72.center) to (75.center);
		\draw (73.center) to (76.center);
		\draw (74.center) to (77.center);
		\draw [bend right=90, looseness=1.75] (86.center) to (87.center);
		\draw (88.center) to (91.center);
		\draw (89.center) to (92.center);
		\draw (90.center) to (93.center);
	\end{pgfonlayer}
\end{tikzpicture}  }     \ .    {  \normalsize  \label{defeqnhalfbraiding}}
\end{align}
\normalsize
As stated in our conventions on the graphical calculus on page~\pageref{congraphcalc}, the cup is the evaluation, and the cap is the coevaluation. 
Thanks to $\mathbb{F}=U\coalg$ and $\mathbb{A}=U\alg$, we see that
$\mathbb{F}$ is coalgebra in $\cat{C}$ while $\mathbb{A}$ an algebra in $\cat{C}$.
The coproduct $\delta : \mathbb{F}\to\mathbb{F}\otimes\mathbb{F}$ of $\mathbb{F}$ can be described as being induced by the coevaluation
\begin{align}
	X^\vee \otimes X\ra{X^\vee \otimes b_X \otimes X} X^\vee \otimes X\otimes X^\vee \otimes X \ ; \label{eqncomultdelta}
\end{align}
dually, the 
product $\gamma : \mathbb{A}\otimes \mathbb{A}\to\mathbb{A}$
can be described as being induced by the 
 evaluation 
\begin{align}
	X \otimes X^\vee \otimes X\otimes X^\vee \ra{X\otimes d_X\otimes X^\vee} X\otimes X^\vee \ . \label{eqngamma}
\end{align}
We have ${^\vee\coalg}\cong \alg$ (and ${^\vee \mathbb{F}} \cong \mathbb{A}$), and under this duality, the product 
 on $\alg$ (and $\mathbb{A}$) and the coproduct on $\coalg$ (and $\mathbb{F}$) translate into each other.
 
 We will denote by $D$ the 
 distinguished invertible object of $\cat{C}$
 that controls  by \cite{eno-d} the quadruple dual of a finite tensor category through the \emph{Radford formula}
$ -^{\vee\vee\vee\vee}\cong D\otimes-\otimes D^{-1} $
that generalizes the classical result on the quadruple of the antipode of a Hopf algebra \cite{radford}.
By \cite[Lemma~5.5]{shimizuunimodular}
$D$ is the socle of the projective cover of the monoidal unit. 
By \cite[Lemma~4.7 \& Theorem~4.10]{shimizuunimodular} there are canonical natural isomorphisms
\begin{align}
	L(D\otimes-) \cong R \cong L(-\otimes D) \ ,  \quad R(D^{-1}\otimes -)\cong L \cong R (-\otimes D^{-1})  \ .    \label{eqnLandR}
\end{align}
One calls a finite tensor category \emph{unimodular} if $D\cong I$. 
By \cite[Theorem~4.10]{shimizuunimodular} a finite tensor category is unimodular if and only if $L\cong R$.

\section{The $E_2$-structure on homotopy invariants\label{secinvariantscommalg}}
In this section,
we present our result for the systematic construction
of an $E_2$-algebra structure on homotopy invariants of braided commutative algebras. 
This will rely heavily on the notions of a braided operad 
whose definition we briefly recall from \cite[Section~5.1]{Fresse1}:
Let $\cat{M}$ be a closed and bicomplete symmetric monoidal category. A \emph{braided operad} $\cat{O}$ consists of objects $\cat{O}(n)\in\cat{M}$ of $n$-ary operations, where $n\ge 0$, that carry an action of the braid group $B_n$ on $n$ strands
(as commonplace in the theory of operads, we will work with \emph{right} actions throughout), a unit $I\to \cat{O}(1)$ for 1-ary operations (here $I$ is the unit of $\cat{M}$) and composition maps
\begin{align}
	\cat{O}(n) \otimes \cat{O}(m_1)\otimes\dots \otimes \cat{O}(m_n) \to \cat{O}(m_1+\dots+m_n) 
\end{align}
such that the composition is associative, unital and compatible with the braid group actions. Morphisms of braided operads are defined analogously to the symmetric case.
There is an obvious restriction functor $\catf{Res}: \catf{SymOp}(\cat{M})\to \catf{BrOp}(\cat{M})$ from symmetric operads in $\cat{M}$ to braided operads in $\cat{M}$ (it restricts arity-wise along the epimorphisms from braid groups to symmetric groups). This restriction functor has a left adjoint, the \emph{symmetrization}
\begin{align}\label{eqn:adjunctionoperads}
	\xymatrix{
		\catf{Sym} \,:\, \catf{BrOp}(\cat{M})        ~\ar@<0.5ex>[r]&\ar@<0.5ex>[l]  ~\catf{SymOp}(\cat{M})  \,:\, \catf{Res} \ , 
	}
\end{align}
that takes arity-wise orbits of pure braid group actions, i.e.\ $(\catf{Sym}\cat{O})(n)=\cat{O}(n)/P_n$ for $n\ge 0$, where $P_n$ is the pure braid group on $n$ strands.

Let $\cat{B}$ be an $\cat{M}$-enriched braided monoidal category. Then for any 
$A\in \cat{B}$, the objects
$\End_A(n) := [A^{\otimes n},A]$ for $n\ge 0$ (we denote here by $[-,-]$ the $\cat{M}$-valued hom)
form a braided operad in $\cat{M}$, the \emph{braided endomorphism operad of $A$}.
If $\cat{O}$ is a braided operad in $\cat{M}$, a \emph{braided $\cat{O}$-algebra in $\cat{B}$} is an object $A\in\cat{B}$ and a map $\cat{O}\to \End_A$ of braided operads.

A \emph{braided commutative algebra} $\somealg \in \cat{B}$ in a braided monoidal category $\cat{B}$ is an algebra whose multiplication $\mu :\somealg \otimes\somealg\to \somealg$ satisfies $ \mu \circ c_{\somealg,\somealg}=\mu$, where $c_{\somealg,\somealg}:\somealg\otimes \somealg\to \somealg\otimes \somealg$ is the braiding. 
Although we will treat general braided commutative algebras (and braided commutative coalgebras which are defined dually), it is instructive to think of the example of the canonical algebra and the canonical coalgebra in the Drinfeld center that were defined in~Section~\ref{firstsection}:

\begin{lemma}[Davydov, Müger, Nikshych,  Ostrik $\text{\cite[Lemma~3.5]{dmno}}$]\label{lemmadmno}
	For any finite tensor category $\cat{C}$,	the canonical algebra $\alg\in Z(\cat{C})$ is braided commutative and the canonical coalgebra $\coalg \in Z(\cat{C})$ is braided cocommutative.  
\end{lemma}

In \cite{dmno} this Lemma is only given in the semisimple case, but the argument can be extended to the non-semisimple case as well, see \cite[Appendix~A.3]{shimizucoend} for a formulation in terms of module categories that also covers the situation as given in Lemma~\ref{lemmadmno}.

For a  braided finite tensor category $\cat{B}$ 
and a braided commutative algebra $\somealg \in \cat{B}$, denote by
$\iota : \somealg \to \isomealg$ an injective resolution of $\somealg$ in $\cat{B}$; the map $\iota$ will be referred to as \emph{coaugmentation}.
For $n\ge 0$, we define the  map
\begin{align}  k \to      \cat{B}(\somealg^{\otimes n},\isomealg)  \label{eqnmapBn}
\end{align}
that selects the map $\somealg^{\otimes n} \ra{\mu } \somealg \ra{\iota} \isomealg$ defined as the concatenation of the ($n$-fold) multiplication $\mu$ of $\somealg$ with the coaugmentation $\iota$ of the injective resolution.
The map itself is not very interesting, but the non-trivial point is that it is actually a map of chain complexes \emph{with $B_n$-action}, where the $B_n$-action on $k$ is trivial and the one on $\cat{B}(\somealg^{\otimes n},\isomealg)$ comes by virtue of $\cat{B}$ being a braided category. 
The fact that \eqref{eqnmapBn} is  really $B_n$-equivariant follows from the assumption that $\somealg$ is braided commutative.

Using again that $\cat{B}$ is braided,
the mapping complex $\cat{B}(\isomealg^{\otimes n},\isomealg)$
comes with a $B_n$-action.
The precomposition with the map $\iota^{\otimes n}:\somealg^{\otimes n}\to \isomealg^{\otimes n}$
yields a map
\begin{align}\label{eqntrivialfibration}
	\left(   \iota^{\otimes n}  \right)^* : \cat{B}(\isomealg^{\otimes n},\isomealg)\to \cat{B}( \somealg^{\otimes n},\isomealg) \ . 
\end{align}
This map is also $B_n$-equivariant.
We may now define the chain complex $J_\somealg(n)$ as a pullback
\begin{equation}\label{eqncheckDdefined}
	\begin{tikzcd}
		\ar[rr] \ar[dd] J_\somealg(n) &&  \cat{B}(\isomealg^{\otimes n},\isomealg)   \ar[dd, "\left(   \iota^{\otimes n}  \right)^*"]  \\ \\
		k 	 \ar[rr,swap,"\eqref{eqnmapBn}"]    && \cat{B}(\somealg^{\otimes n},\isomealg	) \ ,  \\
	\end{tikzcd}
\end{equation}
of differential graded $k$-vector spaces with $B_n$-action (it will be explained in the proof of Proposition~\ref{propJ} below that $\left(   \iota^{\otimes n}  \right)^*$ is, in particular, a fibration, so that $J_\somealg(n)$ is also a homotopy pullback).

\begin{proposition}\label{propJ}
	Let $\cat{B}$ be a braided finite tensor category 
	and $\somealg\in \cat{B}$
	a
	braided commutative algebra.
	The chain complexes $J_\somealg(n)$ defined  in \eqref{eqncheckDdefined} for $n\ge 0$
	naturally form a braided operad $J_\somealg$ in differential graded $k$-vector spaces such that the following holds:
	\begin{pnum}
		\item The operad $J_\somealg$ is acyclic in the sense that it comes with a canonical trivial fibration $J_\somealg \to k$, i.e.\
		$J_\somealg$ is a model for the braided commutative operad.
		\label{propJ1}
		\item The maps $J_\somealg(n)\to \cat{B}(\isomealg^{\otimes n},\isomealg)$ from \eqref{eqncheckDdefined} endow $\isomealg$ with the structure of a braided $J_\somealg$-algebra.\label{propJ2}
	\end{pnum}
\end{proposition}

\begin{proof} We first establish the braided operad structure:
	The complexes $J_\somealg(n)$ come with an $B_n$-action by definition. Moreover, the identity of $\isomealg$ seen as map $k\to \cat{B}(\isomealg,\isomealg)$ yields a map $k \to J_\somealg(1)$ that we define as a unit.
	In order to define the operadic composition, let $m_1,\dots,m_n \ge 0$ be given. By definition of $J_\somealg$ we obtain the maps~$(*)$ and~$(**)$ in the following diagram: 
	\small
	\begin{equation}
		\begin{tikzcd}
			J_\somealg(n)\otimes \bigotimes_{j=1}^n J_\somealg(m_j) \ar[rrrd, bend left=20,"(**)"] \ar[rdddd,bend right=40,swap,"(*)"] \ar[rdd,dashed, "\exists \, !"] \\ &&& \ar[d,"\text{composition in $\cat{B}$}"]\cat{B}(\isomealg^{\otimes n},\isomealg	)\otimes \bigotimes_{j=1}^n \cat{B}(\isomealg^{\otimes m_j},\isomealg	) \\ &	\ar[rr] \ar[dd] J_\somealg(m_1+\dots+m_n) &&  \cat{B}(\isomealg^{\otimes (m_1+\dots+m_n)},\isomealg)   \ar[dd, "\left(   \iota^{\otimes (m_1+\dots+m_n)}  \right)^*"]  \\ \\
			& 	k 	 \ar[rr]    && \cat{B}(\somealg^{\otimes (m_1+\dots+m_n)},\isomealg	) \ .  \\
		\end{tikzcd}
	\end{equation}
	\normalsize
	It is straightforward to see that the outer pentagon commutes. Now by the universal property of the pullback there is unique map $J_\somealg(n)\otimes \bigotimes_{j=1}^n J_\somealg(m_j)\to J_\somealg(m_1+\dots+m_n)$ making the entire diagram commute. We define this to be the needed operadic composition map. 
	The composition can be seen to be equivariant. Since it is induced by composition in $\cat{B}$, it is associative and unital with respect to the identity (which we defined as operadic unit).
	The operad structure on $J_\somealg$ is defined in such a way that statement~\ref{propJ2} holds by construction. 
	
	It remains to prove~\ref{propJ1}: First observe that 
	the maps $J_\somealg(n)\to k$ are $B_n$-equivariant by construction and are also compatible with composition. Therefore, we only need to show that $J_\somealg(n)\to k$ for fixed $n\ge 0$ is a trivial fibration.
	Indeed, the exactness of the monoidal product in $\cat{B}$
	ensures that $\iota^{\otimes n}:\somealg^{\otimes n}\to\isomealg^{\otimes n}$ is again an injective resolution, i.e.\  a trivial cofibration in the injective model structure on  complexes in $\cat{B}$.
	 Since $\isomealg$ is fibrant in this model structure, the precomposition with $\iota^{\otimes n}$ in \eqref{eqntrivialfibration} is a trivial fibration.  
	Now $J_\somealg(n)\to k$, as the pullback of a trivial fibration according to its definition in
	\eqref{eqncheckDdefined},
	is a trivial fibration as well. 
\end{proof}

The construction from Proposition~\ref{propJ} can be used for the construction of differential graded $E_2$-algebras. In order to see this, let us record the following two straightforward Lemmas.

It is well-known that  a symmetric lax monoidal functor preserves operadic algebras. The following Lemma is a braided version of this fact:

\begin{lemma}\label{lemmabop1}
	Let $F:\cat{B}\to\cat{B}'$ be an enriched braided lax monoidal functor between braided monoidal categories enriched over $\cat{M}$ and let $\cat{O}$ be a braided operad in $\cat{M}$.
	Then for any braided $\cat{O}$-algebra $A$ in $\cat{B}$, the image $F(A)$ naturally comes with the structure of a braided $\cat{O}$-algebra.
\end{lemma}

\begin{proof}
	The  structure maps that turn $F(A)$ into a braided $\cat{O}$-algebra are
	\begin{align} \cat{O}(n) \to [A^{\otimes n},A] \ra{F} \left[  F\left(   A^{\otimes n}\right) ,F(A) \right] \to \left[  \left( F\left(   A\right) \right)^{\otimes n} ,F(A) \right] \ , \quad n\ge 0 \ .  \end{align} 
	The first map is the structure map of $A$, the third map precomposes with the maps $\left( F(A)\right)^{\otimes n} \to F\left(   A^{\otimes n}\right)$ that are a part of the lax monoidal structure of $F$. These maps are $B_n$-equivariant because $F$ is braided. 
\end{proof}

\begin{lemma}\label{lemmabop2}
	Let $\cat{O}$ be a braided operad in $\cat{M}$ and $A$ a braided $\cat{O}$-algebra in a {\normalfont symmetric} monoidal category $\cat{B}$ enriched over $\cat{M}$. Then $A$ induces in a canonical way an algebra over the symmetrization $\catf{Sym}\, \cat{O}$ of $\cat{O}$. More precisely, the structure map $\cat{O}\to \End_A$ canonically factors through the unit $\cat{O}\to \catf{Res} \,\catf{Sym}\, \cat{O}$ of the adjunction $\catf{Sym}\dashv \catf{Res}$ from~\eqref{eqn:adjunctionoperads}. 
\end{lemma}

\begin{proof}
	Pure braid group elements act trivially on $[A^{\otimes n},A]$ because $\cat{B}$ is symmetric. As a consequence,
	the structure maps of $A$ factor as
	\begin{equation}
		\begin{tikzcd}
			\ar[rr] \ar[rd] \cat{O}(n) &&  \text{$[A^{\otimes n},A]$} \ .      \\ 
			& \catf{Sym}\, \cat{O} (n) = \cat{O}(n)/P_n  \ar[ur,dashed]  \\
		\end{tikzcd}
	\end{equation}
	This implies the assertion. 
\end{proof}

\begin{proposition}\label{propE2onhom}
	Let $\cat{B}$ be a  braided finite tensor category.
	Then for any braided commutative algebra $\somealg \in \cat{B}$ 
	and any braided cocommutative coalgebra $\somecoalg \in \cat{B}$
	the derived morphism space $\cat{B}(\somecoalg,\isomealg)$ is naturally an $E_2$-algebra. 
\end{proposition}

\begin{proof}
	We denote by $\check J_\somealg$ a resolution of the braided operad $J_\somealg$ 
	 that is arity-wise a  projectively cofibrant $k[B_n]$-module
	 (a cofibrant object in the projective model structure on differential graded $k[B_n]$-modules); this is the braided analogue of a $\Sigma$-cofibrant symmetric operad. In other words, we pick a braided operad $\check J_\somealg$ with arity-wise projectively cofibrant braid group actions 
	and a trivial fibration $\check J_\somealg\to J_\somealg$.

	Since $J_\somealg$ and thus $\check J_\somealg$ is acyclic thanks to Proposition~\ref{propJ}~\ref{propJ1} and since $\check J_\somealg$ has a projectively cofibrant braid group action, we obtain
	\begin{align}  (\catf{Sym}\, \check J_\somealg)(n)=\left(  \check J_\somealg  (n)\right) / P_n = C_*(BP_n;k) \simeq C_*(E_2(n);k) \ ,
	\end{align} where $C_*(-;k)$ is the functor taking $k$-chains. 
 Consequently,	we obtain an equivalence 
	\begin{align}  \catf{Sym}\, \check J_\somealg\simeq C_*(E_2;k) \  \label{eqnsymcheckD}
	\end{align} of operads. In other words, $\catf{Sym}\, \check J_\somealg$ is a model for $E_2$. This is an instance of the Recognition Principle for $E_2$ \cite{fie96}.

	If we pull back the $J_\somealg$-action on $\isomealg$ from Proposition~\ref{propJ}~\ref{propJ2} along $\check J_\somealg\to J_\somealg$, we turn $\isomealg$ into  a braided $\check J_\somealg$-algebra. 
	
	The functor $\cat{B}(\somecoalg,-):\catf{Ch}(\cat{B})\to \Ch$ is 
	\begin{itemize}
		\item lax monoidal since $\somecoalg$ is a coalgebra,
		\item and also braided  since $\somecoalg$ is braided cocommutative. 
	\end{itemize}
	This implies by Lemma~\ref{lemmabop1} that $\cat{B}(\somecoalg,\isomealg)$ becomes a braided $\check J_\somealg$-algebra, which by Lemma~\ref{lemmabop2} induces a $\catf{Sym}\, \check J_\somealg$-algebra structure on $\cat{B}(\somecoalg,\isomealg)$ because $\Ch$ is symmetric. 
	Now \eqref{eqnsymcheckD} yields the assertion.
\end{proof}

We apply Proposition~\ref{propE2onhom} to the canonical coalgebra $\coalg\in Z(\cat{C})$ to obtain a source of $E_2$-algebras:

\begin{theorem}\label{thmE2derivedhom}
	Let $\mathbb{T} \in \cat{C}$ be an algebra in a finite tensor category $\cat{C}$
	together with a lift to a braided commutative algebra $\somealg \in Z(\cat{C})$ in the Drinfeld center.
	Then the multiplication of $\mathbb{T}$ and the half braiding of $\mathbb{T}$ induce the structure of an $E_2$-algebra on the space $\cat{C}(I,\mathbb{T}^\bullet)$ of homotopy invariants of $\mathbb{T}$. 
\end{theorem}

\begin{proof}
	By assumption we have $\mathbb{T}=U\somealg$ as algebras with the forgetful functor $U:Z(\cat{C})\to \cat{C}$, where $\somealg$ is braided commutative. 
	Now let $\coalg$ be the canonical coalgebra in $Z(\cat{C})$, namely the image $LI$ of the unit $I\in\cat{C}$ under the oplax monoidal left adjoint $L:\cat{C}\to Z(\cat{C})$ to $U$. Since $\coalg$ is braided cocommutative by Lemma~\ref{lemmadmno}, we may apply Proposition~\ref{propE2onhom} and find that $Z(\cat{C})(\coalg,\isomealg)$ comes with an $E_2$-structure. 
	
	Using the adjunction $L\dashv U$, 
	we observe
	\begin{align} Z(\cat{C})(\coalg,\isomealg) = Z(\cat{C})(LI,\isomealg)\cong \cat{C}(I,U \isomealg) \ .  \end{align}
	Therefore, $\cat{C}(I,U \isomealg)$ inherits the $E_2$-structure.
	
	It remains to show that $U \isomealg$ is an injective resolution of $\mathbb{T}$: Since $U$ is exact, the map $\mathbb{T}=U\somealg \to U \isomealg$ is a monomorphism and an equivalence. Finally, $U \isomealg$ is also degree-wise injective because $U$ is a right adjoint whose left adjoint $L$ is exact by \cite[Corollary~4.9]{shimizuunimodular} and hence preserves injective objects.
\end{proof}

Let us demonstrate the use of Theorem~\ref{thmE2derivedhom}:
For any braided finite tensor category $\cat{C}$,
denote by $\bar{\cat{C}}$ the same finite tensor category, but equipped with the inverse braiding. The Deligne product $\bar{\cat{C}}\boxtimes {\cat{C}}$ is a braided finite tensor category again. There is now a canonical braided monoidal functor 
$
G:\bar{\cat{C}}\boxtimes\cat{C}\to Z(\cat{C}) 
$ \cite{eno-d}
that sends $X\boxtimes Y \in \bar{\cat{C}}\boxtimes\cat{C}$
to $X\otimes Y$ together with the half braiding whose component indexed by some $Z\in \cat{C}$ is given by
\begin{align}
X\otimes Y \otimes Z \ra{X\otimes c_{Y,Z}} X\otimes Z \otimes Y \ra{c_{X,Z}^{-1}\otimes Y} Z\otimes X\otimes Y\ .\end{align} 
This half braiding is sometimes called the \emph{field goal transform} or the \emph{dolphin half braiding}.
The structure maps that turn $G$ into a braided monoidal functor use the braiding of $Z(\cat{C})$. 
For any braided commutative algebra $B \in \bar{\cat{C}}\boxtimes\cat{C}$, the object $G(B)$ is naturally a braided commutative algebra 
in $Z(\cat{C})$ that lifts the algebra $B_\otimes :=UG(B)$ to the Drinfeld center (the notation $B_\otimes$ is explained by the fact that $B_\otimes$ by definition is obtained by applying the monoidal product to $B$).  As an immediate consequence of Theorem~\ref{thmE2derivedhom},
we obtain:

\begin{corollary}\label{corbraided}
Let $\cat{C}$ be a braided finite tensor category.
For any braided commutative algebra $B\in\bar{\cat{C}}\boxtimes\cat{C}$, denote by $B_\otimes$ the algebra in $\cat{C}$ obtained by applying the monoidal product functor to $B$. 
Then the homotopy invariants $\cat{C}(I,B_\otimes^\bullet)$ of $B_\otimes$ naturally form an $E_2$-algebra. 
\end{corollary}

\begin{example}\label{exampledolphin}
If we consider the special case $B=\int^{X\in\cat{C}}X^\vee \boxtimes X$ with the structure of a braided commutative algebra from \cite[Proposition~2.3]{fss},
then $B_\otimes$ is the canonical coend $\mathbb{F}=\int^{X\in\cat{C}}X^\vee \otimes X$ with the underlying algebra structure given by Lyubashenko \cite[Section~2]{lyu}. 
The resulting braided commutative algebra in $Z(\cat{C})$ is described directly, without passing through an algebra in $\bar{\cat{C}}\boxtimes\cat{C}$, by Neuchl and Schauenburg in~\cite[Definition~4 \& Proposition~5]{neuchlschauenburg}.
A detailed description of how the braided commutative algebra structure on $\mathbb{F}$ is obtained by applying $\otimes$ to $B=\int^{X\in\cat{C}}X^\vee \boxtimes X$ is given in 
\cite[Section~2]{cardycase}.
 If $\cat{C}$ is given by  finite-dimensional modules over a finite-dimensional quasi-triangular Hopf algebra, then $\mathbb{F}$ is the dual of the Hopf algebra with coadjoint action (we discuss the dual version of this statement in Example~\ref{exquantumgroups}). 
The resulting $E_2$-algebra $\cat{C}(I,\mathbb{F}^\bullet)$
is needed as an important technical ingredient for proof of the differential graded Verlinde algebra \cite{vd}. It appears there under the name \emph{dolphin algebra} (because the the $E_2$-algebra structure comes from the dolphin half braiding). 
\end{example}

The construction of Theorem~\ref{thmE2derivedhom}
is natural with respect to the input datum
in the following sense:

\begin{proposition}[Naturality in the braided algebra]\label{propnaturality}
	Let $\cat{C}$ be a finite tensor category with algebras $\mathbb{T}$ and $\mathbb{U}$ in $\cat{C}$ with lifts 
	$\somealg$ and $\catf{U} \in Z(\cat{C})$ to braided commutative algebras.
	Then any algebra map $\varphi: \mathbb{T}\to \mathbb{U}$, which has the property to induce a map $\somealg \to \catf{U}$ in the Drinfeld center, gives rise to a map
	\begin{align} \varphi^\bullet : \cat{C}(I,\mathbb{T}^\bullet )\to \cat{C}(I,\mathbb{U}^\bullet)
	\end{align}
	of $E_2$-algebras.
\end{proposition}

\begin{proof}
	The assumption says exactly that $\varphi:\somealg \to \catf{U}$ gives us a map of algebras in $Z(\cat{C})$.
	We can extend $\varphi$ to a map $\varphi^\bullet : \isomealg \to \catf{U}^\bullet$ between injective resolutions such that 
	$\varphi^\bullet \circ \iota_{\somealg} = \iota_{\catf{U}} \circ \varphi$ for the coaugmentations $\iota_{\somealg}:\somealg\to\isomealg$ and $\iota_{\catf{U}} : \catf{U}\to\catf{U}^\bullet$.
	The idea is to see $\varphi^\bullet : \isomealg \to \catf{U}^\bullet$ as map  of braided algebras over an acyclic braided operad and then to apply the functor $Z(\cat{C})(\coalg,-)$ as in the proof of Theorem~\ref{thmE2derivedhom}. This will give us a map of $E_2$-algebras.
	
	The non-obvious point is \emph{how} to see $\varphi^\bullet : \isomealg \to \catf{U}^\bullet$ as map of braided algebras over an acyclic braided operad because, according to Proposition~\ref{propJ},
	$\isomealg$ is a $J_\somealg$-algebra while $\catf{U}^\bullet$ is a $J_\catf{U}$-algebra. Both algebras are defined over different operads. Moreover, $\varphi$ does \emph{not} directly induce a map $J_\somealg\to J_\catf{U}$.

	Instead,
	we use the following construction:
	First we define  for $n\ge 0$ the pullback of complexes with $B_n$-action
	\begin{equation}
		\begin{tikzcd}
			\ar[rrrrrr] \ar[dd] K_\varphi(n) &&&& && Z(\cat{C})(\isomealg^{\otimes n},\catf{U}^\bullet)   \ar[dd, "  (\iota_\somealg^{\otimes n})^*     "]  \\ \\
			k	 \ar[rrrrrr,swap," 1 \mapsto \left( \somealg^{\otimes} \xrightarrow{\mu^n_\somealg} \somealg  \xrightarrow{\varphi} \catf{U} \xrightarrow{\iota_\catf{U}} \catf{U}^\bullet \right)  "]    &&&&&& Z(\cat{C})(\somealg^{\otimes n},\catf{U}^\bullet)   \ .  \\
		\end{tikzcd}\label{pullbackpullback}
	\end{equation}
	As in the proof of Proposition~\ref{propJ}, this is also a homotopy pullback.
	Since the right vertical map is a trivial fibration, so is $K_\varphi(n)\to k$. 
	The map
	\begin{align}
		J_\somealg(n) \to Z(\cat{C})(\isomealg^{\otimes n},\isomealg ) \ra{\varphi^\bullet_*}  Z(\cat{C})(\isomealg^{\otimes n},\catf{U}^\bullet) 
	\end{align}
	induces by construction a map $J_\somealg(n)\to K_\varphi(n)$ that by abuse of notation we just write as $\varphi_*$.
	Similarly,
	the map
	\begin{align}
		J_\catf{U}(n) \to Z(\cat{C})({\catf{U}^\bullet}^{\otimes n},\catf{U}^\bullet) \ra{({\varphi^\bullet}^{\otimes n})^*}  Z(\cat{C})(\isomealg^{\otimes n},\catf{U}^\bullet) 
	\end{align}
	induces a map $\varphi^* : J_\catf{U}(n) \to K_\varphi(n)$. 
	This uses
	\begin{align}
		j \circ {\varphi^\bullet}^{\otimes n} \circ \iota_\somealg^{\otimes n} = j \circ  \iota_{ \catf{U}}^{\otimes n} \circ \varphi^{\otimes n} = \iota_{ \catf{U}} \circ \mu_\catf{U}^n \circ \varphi^{\otimes n} =\iota_{ \catf{U}} \circ  \varphi \circ \mu_\somealg^n  \quad \text{for}\quad j \in 	J_\catf{U}(n) \ , 
	\end{align}
	where $\varphi$ being a map of algebras enters in the last step.
	The maps $\varphi_* : J_\somealg(n)\to K_\varphi(n)$ and $\varphi^* : J_\catf{U}(n) \to K_\varphi(n)$ are not zero because $\varphi$ must preserve unit and hence is not zero (we assume that the algebras $\mathbb{T}$ and $\mathbb{U}$ are not zero). Since $J_\somealg(n)$, $J_\catf{U}(n)$ and $ K_\varphi(n)$
	are acyclic, $\varphi_*$ and $\varphi^*$ are equivalences. 
	As a consequence,
	the homotopy pullback 
	\begin{equation}
		\begin{tikzcd}
			\ar[rr] \ar[dd] (J_\somealg \times _\varphi J_\catf{U})(n) &&  J_\somealg(n)   \ar[dd, "\varphi_*"]  \\ \\
			J_\catf{U}(n)	 \ar[rr,swap,"\varphi^*"]    && K_\varphi(n)   \\
		\end{tikzcd}\label{pullbackpullback}
	\end{equation}
	is also acyclic.
	The complexes $(J_\somealg \times _\varphi J_\catf{U})(n)$ form an acyclic braided operad $J_\somealg \times _\varphi J_\catf{U}$.
	By virtue of the projections,
	$J_\somealg \times _\varphi J_\catf{U}\to J_\somealg$ 
	and $J_\somealg \times _\varphi J_\catf{U}\to J_\catf{U}$,
	$\isomealg$ and $\catf{U}^\bullet$ become braided $J_\somealg \times _\varphi J_\catf{U}$-algebras such that $\varphi^\bullet : \isomealg \to \catf{U}^\bullet$ becomes a $J_\somealg \times _\varphi J_\catf{U}$-algebra map (up to coherent homotopy depending on the model for \eqref{pullbackpullback}). 
\end{proof}

\section{Resolving the canonical algebra\label{firstsection}}
We intend to eventually apply Theorem~\ref{thmE2derivedhom} to the canonical algebra $\mathbb{A}$ and to exploit the relation to Hochschild cochains.
As a preparation for the very technical next section,
we present in this section an injective resolution of $\mathbb{A}$ which conveniently allows us to do that.

For any finite tensor category $\cat{C}$, $X\in\cat{C}$ and a finite-dimensional vector space $W$,
recall that we denote by $X^W=W^*\otimes X$  the powering of  $X$ by $W$.
We may now define the objects 
\begin{align}
	\prod_{X_0,\dots,X_n \in \Proj\cat{C}}    (X_0\otimes X_n^\vee)^{\cat{C}(X_n,X_{n-1})\otimes\dots\otimes\cat{C}(X_1,X_0)}        \label{eqnterminininj}
\end{align}
that we  organize into a cosimplicial object
\begin{equation}\label{eqnresA}
	\begin{tikzcd}
		\displaystyle
		\prod_{X_0\in \Proj \cat{C}} X_0\otimes X_0^\vee \ar[r, shift left=2] \ar[r, shift right=2]
		& \displaystyle \prod_{X_0,X_1 \in \Proj \cat{C}}       (X_0\otimes X_1^\vee)  ^{  \cat{C}(X_1,X_0)  }  \ar[r, shift left=4] \ar[r, shift right=4]\ar[r]  \  
		\ar[l] & \ar[l,shift left=2] \ar[l,shift right = 2]\displaystyle \dots \ .   \\
	\end{tikzcd}
\end{equation}
which  a priori lives in a completion of $\cat{C}$
by infinite products.
We denote by $\frint_{X\in\Proj \cat{C}} X \otimes X^\vee$ the differential graded object in $\cat{C}$ which is obtained as the totalization of the restriction to any finite collection of projective objects in $\cat{C}$ that contains a projective generator.

\begin{lemma}\label{lemmaresA}
	Let $\cat{C}$ be a finite tensor category.
	The differential graded object 
	$\frint_{X\in\Proj \cat{C}} X \otimes X^\vee$ in $\cat{C}$ is well-defined up to equivalence and an injective resolution of the canonical end $\mathbb{A}=\int_X X\otimes X^\vee$.
\end{lemma}

\begin{proof}
	We apply the duality functor ${^\vee-}$ to the projective resolution $\flint^{X\in\Proj\cat{C}}X^\vee \otimes X$ of $\mathbb{F}$ given in \cite[Corollary~4.3]{dva} and observe that, after a substitution of dummy variables, we obtain~\eqref{eqnresA}. 
	Since $^\vee\mathbb{F} \cong \mathbb{A}$, we now obtain the well-definedness of $\frint_{X\in\Proj \cat{C}} X \otimes X^\vee$ and the fact that it is an injective resolution of $\mathbb{A}$ (all of this crucially uses that in a finite tensor category the projective objects are precisely the injective ones).
\end{proof}

Let $X_0,\dots,X_{p+q}$ be a family of projective objects in $\cat{C}$ from the finite collection of projective objects
used to define $\frint_{X\in\Proj \cat{C}} X \otimes X^\vee$.
The evaluation $d_{X_p}: X_p^\vee\otimes X_p\to I$ 
of $X_p$ 
induces a map

\small
\begin{align} \label{eqnauxmapmult}\begin{array}{c} \left(  X_0 \otimes X_p^\vee \right)^{\cat{C}(X_1,X_0)\otimes\dots\otimes \cat{C}(X_{p},X_{p-1})} \\ \otimes \left(  X_p \otimes X_{p+q}^\vee \right)^{\cat{C}(X_p,X_{p+1})  \otimes\dots\otimes\cat{C}(X_{p+q-1},X_{p+q})}\end{array} \ \  \ra{ \ \ \ } \ \   \left(  X_0 \otimes X_{p+q}^\vee \right) ^{\cat{C}(X_1,X_0)\otimes\dots\otimes \cat{C}(X_{p+q},X_{p+q-1})}  \ . \end{align} \normalsize We may now define a map \begin{align}\label{eqndefgamma} \gamma^\bullet_{p+q} : \left( \frint_{X\in\Proj \cat{C}} X \otimes X^\vee \right)^p
	\otimes 
	\left( \frint_{X\in\Proj \cat{C}} X \otimes X^\vee \right)^q 
	\to 
	\left( \frint_{X\in\Proj \cat{C}} X \otimes X^\vee \right)^{p+q}  \end{align} as follows:
By the universal property of the product it suffices to give the component for any family
$X_0,\dots,X_{p+q}$ of projective objects in $\cat{C}$. 
We define this component to be the map that projects to the component of $X_0,\dots,X_p$ for the first factor and to $X_p,\dots,X_{p+q}$ for the second factor and applies the map~\eqref{eqnauxmapmult}. 
A direct computation shows that~\eqref{eqndefgamma}
yields a chain map and an associative and unital multiplication $\gamma^\bullet$
on $\frint_{X\in\Proj \cat{C}} X \otimes X^\vee$.
This is a model for a lift of the algebra structure $\gamma : \mathbb{A}\otimes\mathbb{A}\to\mathbb{A}$ from \eqref{eqngamma}
to the injective resolution as one can directly verify:

\begin{lemma}\label{lemmamodelforresA}
	For any finite tensor category $\cat{C}$, the coaugmentation $\mathbb{A}\to \frint_{X\in\Proj \cat{C}} X \otimes X^\vee$ is an equivalence \begin{align}\bspace{\mathbb{A}}{\gamma}\ra{\simeq} \bspace{\frint_{X\in\Proj \cat{C}} X \otimes X^\vee}{\gamma^\bullet}    \label{eqngammacoaug}    \end{align}
	of differential graded algebras.
\end{lemma}

Let us now discuss the relation to Hochschild cochains:
For a  linear category $\cat{A}$ over $k$ (to be thought of as algebra with several objects), the Hochschild cochain complex $\rint_{a\in \cat{A}} \cat{A}(a,a)$ 
is the homotopy end over the endomorphism spaces of objects in $\cat{A}$.
When spelled out, it is
the cochain complex of vector spaces which in cohomological degree $n\ge 0$ is given by
\begin{align}\label{eqnhochschildcomplex}
	\left(\rint_{a\in \cat{A}} \cat{A}(a,a) \right)^n =\left\{ \begin{array}{cl} \prod_{a_0\in\cat{A}} \cat{A}(a_0,a_0) & \text{for $n=0$} \ , \\ \prod_{a_0,\dots,a_n\in\cat{A}} \Hom_k     \left(    \cat{A}(a_1,a_0) \otimes \dots \otimes \cat{A}(a_n,a_{n-1}),\cat{A}(a_n,a_0)       \right) & \text{for $n\ge 1$} \ .\end{array} \right.
\end{align}
The composition in $\cat{A}$ gives us the differentials.
One 
may define the \emph{cup product} $\smile$ on the Hochschild cochain complex as follows:
Let a $p$-cochain $\varphi$ and a $q$-cochain $\psi$ be given, and let
$(a_0,\dots,a_{p+q})$ be a $p+q$-tuple of objects in $\cat{A}$.
The $(a_0,\dots,a_{p+q})$-component $(	\varphi \smile \psi)_{a_0,\dots,a_{p+q}}$ of the cochain $\varphi \smile \psi$
with degree $p+q$ is given by
\begin{align}
	(	\varphi \smile \psi)_{a_0,\dots,a_{p+q}} := \varphi_{a_0,\dots,a_p} \circ_{a_p} \psi_{a_{p},\dots,a_{p+q}} \ , 
\end{align}
where $\circ_{a_p}$ denotes the composition over $a_p$ in $\cat{A}$. 
In cohomology, the cup product yields a graded commutative product.
By work of Gerstenhaber \cite{gerstenhaber} the Hochschild cohomology comes also with a degree one bracket, a structure that today is called \emph{Gerstenhaber algebra} (we will review the definition of the Gerstenhaber bracket once we need it).

For a finite tensor category $\cat{C}$, we apply this definition to $\Proj\cat{C}$ and call 
$\rint_{X\in\Proj \cat{C}}    \cat{C}(X,X)$
the \emph{Hochschild cochain complex of $\cat{C}$}. 
The fact that the end runs over the projective objects will always be implicit.

\begin{proposition}\label{propcocomplex}	
	For any finite tensor category $\cat{C}$, there is a canonical equivalence
	\begin{align} \bspace {\rint_{X\in\Proj \cat{C}}    \cat{C}(X,X)      }{\smile} \ra{\simeq}  \bspace{\cat{C}\left(I,\frint_{X\in\Proj \cat{C}} X\otimes X^\vee \right)}{\gamma^\bullet} 
	\end{align}
	of differential graded algebras.
\end{proposition}

	If we ignore the algebra structure, Proposition~\ref{propcocomplex} reduces (on cohomology) to an isomorphism \begin{align}HH^*(\cat{C})\cong \Ext^*_\cat{C}(I,\mathbb{A})\end{align} which is well-known for Hopf algebras and goes back to Cartan and Eilenberg \cite{cartaneilenberg} as reviewed in \cite[Proposition~2.1]{bichon} and formulated for arbitrary finite tensor categories in \cite[Corollary~7.5]{shimizucoend}. 
	The new aspect in Proposition~\ref{propcocomplex} is the compatibility with the algebra structures
	(and, on a technical level, the fact that we exhibit a very convenient cochain level model that we can work with later).

\begin{proof}[\slshape Proof of Proposition~\ref{propcocomplex}]
	The finite homotopy end $\frint_{X\in\Proj \cat{C}} X\otimes X^\vee$
	runs by definition over a full subcategory of $\Proj \cat{C}$ consisting of a family of finitely many objects that include a projective generator. For this proof, we denote this subcategory by $\cat{F}\subset \Proj\cat{C}$. Now by definition
	\begin{align} \frint_{X\in\Proj \cat{C}} X\otimes X^\vee &= \rint_{X\in\cat{F}} X\otimes X^\vee \ . \label{auxFeqn1}
	\end{align}
	Moreover, we set
	\begin{align}\frint_{X\in\Proj\cat{C}} \cat{C}(X,X) := \rint_{X\in\cat{F}} \cat{C}(X,X)
		=\rint_{X\in\cat{F}} \cat{F}(X,X) \ . 
		\label{auxFeqn2} \end{align}
	\begin{pnum}
		\item We define a map $\rint_{X\in\Proj \cat{C}}    \cat{C}(X,X)\to \cat{C}\left(I,\frint_{X\in\Proj \cat{C}} X \otimes X^\vee\right)$ by the commutativity of the triangle
		\begin{equation}
			\begin{tikzcd}
				\ar[rr,"(*)"] \ar[ddr,swap,"\simeq"] \rint_{X\in\Proj \cat{C}}    \cat{C}(X,X)  &&  \cat{C}\left(I,\frint_{X\in\Proj \cat{C}}   X\otimes X^\vee\right)     \\ \\
				& \frint_{X\in\Proj \cat{C}}    \cat{C}(X,X) \ , 	 \ar[ruu,swap,"\cong"]   \\
			\end{tikzcd}
		\end{equation}
		where 
		the maps appearing in the triangle are defined as follows:
		\begin{itemize}
			
			\item The equivalence on the left comes from the restriction to  $\cat{F}$.
			(This is the dual version of the Agreement Principle of McCarthy \cite{mcarthy} and Keller \cite{keller}, see \cite[Section~2.2]{dva} for a review.)
			
			\item The isomorphism on the right comes, in each degree, from the 
			isomorphisms
			\begin{align}
				&	\prod_{X_0,\dots,X_n\in \cat{F}} \cat{C}(X_n,X_0) ^{\cat{C}(X_1,X_0)\otimes\dots\otimes\cat{C}(X_n,X_{n-1})} \\ \cong\quad  &\cat{C}\left(I,\prod_{X_0,\dots,X_n\in \cat{F}} \left(    X_0\otimes X_n^\vee  \right)^{\cat{C}(X_1,X_0)\otimes\dots\otimes\cat{C}(X_n,X_{n-1})} \right)
			\end{align}
			induced by duality.
		\end{itemize}
		Now clearly $(*)$ is an equivalence.
		
		\item Since $\cat{F}\subset \Proj \cat{C}$ is a full subcategory, the cup product of 
		$\rint_{X\in\Proj\cat{C}}    \cat{C}(X,X)$ restricts to the cup product of
		$\rint_{X\in\cat{F}} \cat{F}(X,X)$. With the notation from~\eqref{auxFeqn2},
		this means that
		\begin{align} \bspace{\rint_{X\in\Proj\cat{C}}    \cat{C}(X,X)}{\smile}\ra{\simeq} \bspace{\frint_{X\in\Proj\cat{C}}    \cat{C}(X,X)}{\smile}
		\end{align} is an equivalence of differential graded algebras. 
		Hence, it remains to prove
		that \begin{align} \bspace{\frint_{X\in\Proj\cat{C}}    \cat{C}(X,X)}{\smile} \cong \bspace{\cat{C}\left(I,\frint_{X\in\Proj \cat{C}} X\otimes X^\vee \right)   }{\gamma^\bullet}
		\end{align}
		is an isomorphism of algebras. 
		After unpacking the definition
		of the multiplications on both sides,
		this means that for $X_0,\dots,X_{p+q}\in\cat{F}$ and the abbreviations \begin{align} V&:=\cat{C}(X_1,\dots,X_0)\otimes\dots\otimes\cat{C}(X_p,X_{p-1}) \ , \\
			W&:= \cat{C}(X_{p+1},X_p)\otimes\dots\otimes \cat{C}(X_{p+q},X_{p+q-1}) \ , 
		\end{align}
		the square
		\begin{equation}
			\begin{tikzcd}
				\ar[rr,"\smile"] \ar[dd,"\cong"] \cat{C}(X_p,X_0)^V \otimes \cat{C}(X_{p+q},X_p)^W  &&  \cat{C}(X_{p+q},X_0)^{V\otimes W}   \ar[dd,"\cong"]  \\ \\
				\cat{C}\left(I,{X_0\otimes X_p^\vee}^V\right)  \otimes   \cat{C}\left(I,{X_p\otimes X_{p+q}^\vee}^W\right) \ar[rr,"d_{X_p}"]   && 	\cat{C}\left(I,{X_0\otimes X_{p+q}^\vee}^{V\otimes W}\right)  \ ,    \\
			\end{tikzcd}
		\end{equation}
		in which the vertical isomorphisms come from duality, commutes. 
		The exponentials are just tensored together in both clockwise and counterclockwise direction. Since the cup product  composes over $X_p$, the commutativity of the square now boils down to the basic equality\small
		\begin{equation}\begin{tikzpicture}[scale=0.8]
	\begin{pgfonlayer}{nodelayer}
		\node [style=flat box 2] (56) at (-22, -1) {$f$};
		\node [style=flat box 2] (57) at (-22, -4) {$g$};
		\node [style=none] (58) at (-22, 1) {};
		\node [style=none] (59) at (-22, -5) {};
		\node [style=none] (60) at (-20, -5) {};
		\node [style=none] (61) at (-20, 1) {};
		\node [style=none] (62) at (-18.5, -2) {$=$};
		\node [style=flat box 2] (63) at (-17, -2) {$f$};
		\node [style=none] (65) at (-17, 1) {};
		\node [style=none] (66) at (-17, -5) {};
		\node [style=none] (67) at (-15, -5) {};
		\node [style=none] (68) at (-15, -1) {};
		\node [style=none] (69) at (-13, -1) {};
		\node [style=flat box 2] (70) at (-13, -2) {$g$};
		\node [style=none] (71) at (-13, -5) {};
		\node [style=none] (72) at (-11, -5) {};
		\node [style=none] (73) at (-11, 1) {};
		\node [style=none] (74) at (-22, 2) {$X_0$};
		\node [style=none] (75) at (-20, 2) {$X_{p+q}^\vee$};
		\node [style=none] (76) at (-17, 2) {$X_0$};
		\node [style=none] (77) at (-11, 2) {$X_{p+q}^\vee$};
		\node [style=none] (78) at (-23, -2.5) {$X_p$};
		\node [style=none] (79) at (-18, -4) {$X_p$};
	\end{pgfonlayer}
	\begin{pgfonlayer}{edgelayer}
		\draw (56) to (57);
		\draw (58.center) to (56);
		\draw (57) to (59.center);
		\draw [bend right=90, looseness=1.50] (59.center) to (60.center);
		\draw (60.center) to (61.center);
		\draw (65.center) to (63);
		\draw [bend right=90, looseness=1.50] (66.center) to (67.center);
		\draw (67.center) to (68.center);
		\draw (66.center) to (63);
		\draw [bend left=90, looseness=1.50] (68.center) to (69.center);
		\draw (69.center) to (70);
		\draw (71.center) to (70);
		\draw [bend right=90, looseness=1.50] (71.center) to (72.center);
		\draw (72.center) to (73.center);
	\end{pgfonlayer}
\end{tikzpicture}
		\end{equation}\normalsize
		of morphisms $I\to X_0\otimes X_{p+q}^\vee$. 
	\end{pnum}
\end{proof}

\spaceplease
\section{Comparison Theorem}
In this section, we prove that Theorem~\ref{thmE2derivedhom} 
produces 
a solution to Deligne's Conjecture:

\begin{theorem}[Comparison Theorem]\label{thhmcomparisondeligne}
		For any finite tensor category $\cat{C}$, 
		the algebra structure on the canonical end $\mathbb{A}=\int_{X\in\cat{C}} X\otimes X^\vee$ and its canonical lift to the Drinfeld center induces  an $E_2$-algebra structure on the 
		homotopy invariants $\cat{C}(I,\mathbb{A}^\bullet)$. Under the equivalence $\cat{C}(I,\mathbb{A}^\bullet)\simeq \rint_{X\in \Proj \cat{C}} \cat{C}(X,X)$, this $E_2$-structure
		provides a solution to Deligne's Conjecture in the sense that it induces the standard Gerstenhaber structure on the Hochschild cohomology of $\cat{C}$.
\end{theorem}

\begin{proof}[\slshape Outline of the strategy of the proof of Theorem~\ref{thhmcomparisondeligne}.]	In order to obtain the $E_2$-structure on $\cat{C}(I,\mathbb{A}^\bullet)$, we specialize 
Theorem~\ref{thmE2derivedhom}
to the algebra $\mathbb{T}=\mathbb{A}\in\cat{C}$. This is possible because $\mathbb{A}$ lifts to a braided commutative algebra in $Z(\cat{C})$ by Lemma~\ref{lemmadmno}.
One can conclude from Proposition~\ref{propcocomplex} that 
the underlying multiplication of this $E_2$-structure on $\cat{C}(I,\mathbb{A}^\bullet)$ translates to the cup product on the Hochschild cochain complex under the equivalence $\cat{C}(I,\mathbb{A}^\bullet)\simeq \rint_{X\in \Proj \cat{C}} \cat{C}(X,X)$.
It remains to prove that the Gerstenhaber bracket that we extract
from the $E_2$-structure on $\cat{C}(I,\mathbb{A}^\bullet)$ constructed via Theorem~\ref{thmE2derivedhom}
yields the standard Gerstenhaber bracket on the Hochschild cohomology $H^* \left(   \rint_{X\in \Proj \cat{C}} \cat{C}(X,X)    \right)$. 

Unfortunately, this part is relatively involved and will occupy the rest of this section: We need to spell out a model for the  homotopy $h$
between the multiplication $\gamma^\bullet_*$ 
on $\cat{C}(I,\mathbb{A}^\bullet)$ coming from the product $\gamma^\bullet : \mathbb{A}^\bullet \otimes\mathbb{A}^\bullet \to\mathbb{A}^\bullet$
and the opposite multiplication
$\gbos$.
Of course, we cannot just exhibit \emph{any} homotopy, but need to compute the specific homotopy that the $E_2$-structure provided by Theorem~\ref{thmE2derivedhom} gives us. Afterwards, we will extract the Gerstenhaber bracket from the homotopy $h$.
Since the Gerstenhaber bracket is an operation on homology, it suffices to compute $h$ up to a higher homotopy.

More concretely, the homotopy $h$ between $\gamma^\bullet_*$ and $\gbos$ that we need to compute
is the evaluation of the map $C_*(E_2(2);k)\otimes \cat{C}(I,\mathbb{A}^\bullet)^{\otimes 2}\to \cat{C}(I,\mathbb{A}^\bullet)$ (that the $E_2$-structure provides for us)
on the 1-chain on $E_2(2)$ given by the path  in the configuration space of two disks shown in Figure~\ref{figE22}. 

\begin{figure}[h]
	\centering
\begingroup%
  \makeatletter%
  \providecommand\color[2][]{%
    \errmessage{(Inkscape) Color is used for the text in Inkscape, but the package 'color.sty' is not loaded}%
    \renewcommand\color[2][]{}%
  }%
  \providecommand\transparent[1]{%
    \errmessage{(Inkscape) Transparency is used (non-zero) for the text in Inkscape, but the package 'transparent.sty' is not loaded}%
    \renewcommand\transparent[1]{}%
  }%
  \providecommand\rotatebox[2]{#2}%
  \newcommand*\fsize{\dimexpr\f@size pt\relax}%
  \newcommand*\lineheight[1]{\fontsize{\fsize}{#1\fsize}\selectfont}%
  \ifx\svgwidth\undefined%
    \setlength{\unitlength}{122.39047212bp}%
    \ifx\svgscale\undefined%
      \relax%
    \else%
      \setlength{\unitlength}{\unitlength * \real{\svgscale}}%
    \fi%
  \else%
    \setlength{\unitlength}{\svgwidth}%
  \fi%
  \global\let\svgwidth\undefined%
  \global\let\svgscale\undefined%
  \makeatother%
  \begin{picture}(1,1.02367127)%
    \lineheight{1}%
    \setlength\tabcolsep{0pt}%
    \put(0,0){\includegraphics[width=\unitlength,page=1]{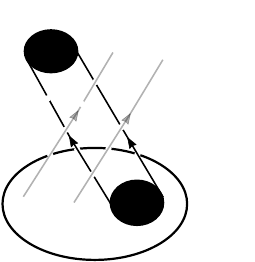}}%
    \put(0.51991552,0.20791529){\color[rgb]{0,0,0}\makebox(0,0)[lt]{\lineheight{1.25}\smash{\begin{tabular}[t]{l}\color{white}2\end{tabular}}}}%
    \put(0,0){\includegraphics[width=\unitlength,page=2]{E2_2-y.pdf}}%
    \put(0.1739865,0.19958574){\color[rgb]{0,0,0}\makebox(0,0)[lt]{\lineheight{1.25}\smash{\begin{tabular}[t]{l}1\end{tabular}}}}%
    \put(0,0){\includegraphics[width=\unitlength,page=3]{E2_2-y.pdf}}%
    \put(0.52376153,0.79208721){\color[rgb]{0,0,0}\makebox(0,0)[lt]{\lineheight{1.25}\smash{\begin{tabular}[t]{l}1\end{tabular}}}}%
    \put(0,0){\includegraphics[width=\unitlength,page=4]{E2_2-y.pdf}}%
    \put(0.17920681,0.80041044){\color[rgb]{0,0,0}\makebox(0,0)[lt]{\lineheight{1.25}\smash{\begin{tabular}[t]{l}\color{white}2\end{tabular}}}}%
    \put(0,0){\includegraphics[width=\unitlength,page=5]{E2_2-y.pdf}}%
  \end{picture}%
\endgroup%

	\caption{The path in $E_2(2)$ which provides for us the homotopy between multiplication and opposite multiplication. 
	The path parameter runs along the vertical axis.}
	\label{figE22}
\end{figure}

The steps of the proof are as follows:
\begin{pnum}
	\item {Construct explicitly with algebraic tools \emph{some} homotopy 
		(that in hindsight we call $h$)
		between the multiplication and the opposite multiplication on $\cat{C}(I,\mathbb{A}^\bullet)$.}
	\label{stepDi}
	\item {Prove that $h$ as constructed in step~\ref{stepDi} agrees up to higher homotopy with the topologically extracted homotopy described above.}
	\label{stepDii}
	
	\item {Extract the Gerstenhaber bracket from $h$ and prove that it agrees with the standard Gerstenhaber bracket on Hochschild cohomology.}
	\label{stepDiii}
\end{pnum}
\end{proof}

\subparagraph{\normalfont \textsc{Step}~\ref{stepDi}.} For 	step~\ref{stepDi},
			we will choose as the 
		injective resolution for $\mathbb{A}$
		the one from Lemma~\ref{lemmaresA}, i.e.\ the (finite) homotopy end
		$\frint_{X\in\Proj\cat{C}} X\otimes X^\vee$.
		With this model, the product
		on $\cat{C}(I,\mathbb{A}^\bullet)$ is given by the product
		$\gamma^\bullet$ of $\mathbb{A}^\bullet$ from~\eqref{eqndefgamma}.
		This has the advantage that $\gamma^\bullet$ translates strictly to the cup product on Hochschild cochains (Proposition~\ref{propcocomplex}).
		
	Let us now begin with the construction of $h$:	
		For $p,q \ge 0$ and $0\le i\le p-1$, we fix 
		an arbitrary family of $p+q$ projective objects
		\begin{align}C_i =(X_0,\dots,X_{i-1},Y_0,\dots,Y_q,X_{i+2},\dots,X_p)\end{align}
		(the labeling is chosen in hindsight and will become clear in a moment)	
		 and
		define the vector spaces
			\begin{align} V' &:= \cat{C}(X_1,X_0)\otimes \dots \otimes \cat{C}(X_{i-1},X_{i-2}) \otimes\cat{C}(Y_0,X_{i-1}) \ ,\label{eqnV1} \\
			W&:= \cat{C}(Y_1,Y_0)\otimes \dots \otimes \cat{C}(Y_q,Y_{q-1}) \ ,\label{eqnV2} \\
			V''&:= \cat{C}(X_{i+2},Y_q)\otimes\dots\otimes \cat{C}(X_p,X_{p-1}) \ .\label{eqnV3}
		\end{align}
		With this notation, the
	component $\left( \mathbb{A}^{p+q-1}\right)^{C_i}$ of the product $\mathbb{A}^{p+q-1}$ (this is the $p+q-1$-th term of $\mathbb{A}^\bullet$, not the $p+q-1$-fold monoidal product) indexed by $C_i$ is given by
		\begin{align}    \left( \mathbb{A}^{p+q-1}\right)^{C_i} = \left(  X_0\otimes X_p^\vee \right)^{V'\otimes  W \otimes V'' } \ , \label{eqnindexedbyCi}
	\end{align}
see~\eqref{eqnterminininj}.
	Similarly, for the  families
		\begin{align}D_i' &:=(X_0,\dots,X_{i-1},Y_0,Y_q,X_{i+2},\dots,X_p) \ ,\\
			D_i'' &:= (Y_0,\dots,Y_q) \ ,   \end{align}
		we have
		\begin{align}    \left( \mathbb{A}^{p}\right)^{D_i'} &= \left(    X_0\otimes X_p^\vee \right)^{V'\otimes \cat{C}(Y_q,Y_0)\otimes   V'' } \ ,\\
			\left( \mathbb{A}^{q}\right)^{D_i''} &= \left(   Y_0 \otimes Y_q^\vee \right)^{W} \ . 
		\end{align}
		Next observe
		\begin{align}
			& 	\cat{C}\left(I,  \left( \mathbb{A}^{p}\right)^{D_i'}   \right)\otimes \cat{C}\left(I,  \left( \mathbb{A}^{q}\right)^{D_i''}   \right)\\ = & \ \Hom_k (  V'\otimes \cat{C}(Y_q,Y_0)\otimes   V'',  \cat{C}(I,X_0 \otimes X_p^\vee)) \otimes \Hom_k(W,  \cat{C}(I,Y_0\otimes Y_q^\vee)  ) \\ \cong  \  &    \Hom_k (  V'\otimes \cat{C}(Y_q,Y_0)\otimes   V'',  \cat{C}(I,X_0 \otimes X_p^\vee)) \otimes \Hom_k(W,  \cat{C}(Y_q,Y_0)  )  \ . 
		\end{align}
		Composition over $ \cat{C}(Y_q,Y_0)$ provides a map
		to  $\Hom_k (  V'\otimes W\otimes   V'',  \cat{C}(I,X_0 \otimes X_p^\vee))$ which is $\cat{C}\left(  I,   \left( \mathbb{A}^{p}\right) ^{C_i} \right)$ by \eqref{eqnindexedbyCi}. Therefore, we obtain a map
		\begin{align}
			\left(  	h_i^{p,q} \right)^{C_i} : \cat{C}\left(I,  \left( \mathbb{A}^{p}\right)^{D_i'}   \right)\otimes \cat{C}\left(I,  \left( \mathbb{A}^{q}\right)^{D_i''}   \right) \to \cat{C}\left(I,  \left( \mathbb{A}^{p+q-1}\right)^{C_i}   \right)
		\end{align}
	decreasing degree by one. 
		By the universal property of the product, we may define the linear map
		\begin{align} h_i^{p,q} : \cat{C}(I,\mathbb{A}^p)\otimes\cat{C}(I,\mathbb{A}^q) \to 
			\cat{C}(I,\mathbb{A}^{p+q-1})
		\end{align}
		by the commutativity of the square
		\begin{equation}
			\begin{tikzcd}
				\ar[rrr,"h_i^{p,q}"] \ar[dd,swap,"\text{projection to components $D_i'$ and $D_i''$}"] \cat{C}(I,\mathbb{A}^p)\otimes\cat{C}(I,\mathbb{A}^q)  && & \cat{C}(I,\mathbb{A}^{p+q-1})   \ar[dd,"\text{projection to component $C_i$}"]  \\ \\
				\cat{C}\left(I,  \left( \mathbb{A}^{p}\right)^{D_i'}   \right)\otimes \cat{C}\left(I,  \left( \mathbb{A}^{q}\right)^{D_i''}   \right)   \ar[rrr,swap,"\left(   	h_i^{p,q} \right)^{C_i}"]   && &\cat{C}\left(   I,	 \left( \mathbb{A}^{p+q-1}\right)^{C_i}\right)  \ .    \\ 
			\end{tikzcd}
		\end{equation}
		and define the components of $h$ by
		\begin{align}
			h^{p,q} := \sum_{i=0}^{p-1}
			(-1)^{i +  (p-1-i)q  }h_i^{p,q} :
			\cat{C}\left(I,  \mathbb{A}^{p}\right)\otimes \cat{C}\left(I,  \mathbb{A}^{q}\right)\to \cat{C}\left(I,  \mathbb{A}^{p+q-1}\right) \ . \label{eqnhsigns}
		\end{align}
		In order to establish a graphical representation for the definition of $h$, we symbolize a $p$-cochain in $\cat{C}(I,\mathbb{A}^\bullet)$, i.e.\
		a vector in
		\begin{align}
			\prod_{X_0,\dots,X_p\in\Proj\cat{C}} \Hom_k \left(   \underbrace{\cat{C}(X_p,X_{p-1})  \otimes \dots \otimes \cat{C}(X_1,X_0)}_{(*)},  \underbrace{\cat{C}(I,   X_0\otimes X_p^\vee) }_{(**)}     \right)     
		\end{align}
		as follows:
		\begin{align} \begin{tikzpicture}
	\begin{pgfonlayer}{nodelayer}
		\node [style=flat box 2] (16) at (0, 0) {{$\varphi$}};
		\node [style=none] (17) at (-0.5, 0.25) {};
		\node [style=none] (18) at (0.5, 0.25) {};
		\node [style=none] (19) at (-1, 2) {};
		\node [style=none] (20) at (1, 2) {};
		\node [style=none] (22) at (0, -1) {$\dots$};
		\node [style=bbox] (23) at (-1, -1) {};
		\node [style=bbox] (24) at (1, -1) {};
	\end{pgfonlayer}
	\begin{pgfonlayer}{edgelayer}
		\draw [in=-90, out=45] (18.center) to (20.center);
		\draw [in=-90, out=150, looseness=0.75] (17.center) to (19.center);
	\end{pgfonlayer}
\end{tikzpicture}    \label{eqnmathbfA3}
		\end{align}
		The box labeled with $\varphi$ with attached legs represents the part $(**)$ of the cochain which is a morphism $I\to X_0 \otimes X_p^\vee$ (we suppress the labels because the cochains have components running over arbitrary labels). 
		The $p$ blank  boxes can be filled with composable morphisms and make the cochains multilinearly dependent on $p$ types of morphisms; this is part $(*)$ of the cochain. 
		With this diagrammatic presentation,
		we arrive at the following description of $h$:
		\begin{align} h_i^{p,q} \left( \  \begin{tikzpicture}
	\begin{pgfonlayer}{nodelayer}
		\node [style=flat box 2] (16) at (0, 0) {{$\varphi$}};
		\node [style=none] (17) at (-0.5, 0.25) {};
		\node [style=none] (18) at (0.5, 0.25) {};
		\node [style=none] (19) at (-1, 2) {};
		\node [style=none] (20) at (1, 2) {};
		\node [style=none] (22) at (0, -1) {$\dots$};
		\node [style=bbox] (23) at (-1, -1) {};
		\node [style=bbox] (24) at (1, -1) {};
	\end{pgfonlayer}
	\begin{pgfonlayer}{edgelayer}
		\draw [in=-90, out=45] (18.center) to (20.center);
		\draw [in=-90, out=150, looseness=0.75] (17.center) to (19.center);
	\end{pgfonlayer}
\end{tikzpicture} \otimes\begin{tikzpicture}
	\begin{pgfonlayer}{nodelayer}
		\node [style=flat box 2] (16) at (0, 0) {{$\psi$}};
		\node [style=none] (17) at (-0.5, 0.25) {};
		\node [style=none] (18) at (0.5, 0.25) {};
		\node [style=none] (19) at (-1, 2) {};
		\node [style=none] (20) at (1, 2) {};
		\node [style=none] (22) at (0, -1) {$\dots$};
		\node [style=rbox] (23) at (-1, -1) {};
		\node [style=rbox] (24) at (1, -1) {};
	\end{pgfonlayer}
	\begin{pgfonlayer}{edgelayer}
		\draw [in=-90, out=45] (18.center) to (20.center);
		\draw [in=-90, out=150, looseness=0.75] (17.center) to (19.center);
	\end{pgfonlayer}
\end{tikzpicture}   \     \right)\quad =  \quad   \begin{tikzpicture}
	\begin{pgfonlayer}{nodelayer}
		\node [style=flat box 2] (16) at (-7.25, 4.5) {{$\varphi$}};
		\node [style=none] (17) at (-7.75, 4.75) {};
		\node [style=none] (18) at (-6.75, 4.75) {};
		\node [style=none] (19) at (-8.25, 6.5) {};
		\node [style=none] (20) at (-6.25, 6.5) {};
		\node [style=none] (22) at (-9, 1) {$\dots$};
		\node [style=flat box 2] (24) at (-7.25, 2.5) {{$\psi$}};
		\node [style=none] (25) at (-7.5, 2.75) {};
		\node [style=none] (26) at (-7, 2.75) {};
		\node [style=none] (27) at (-7.5, 3.75) {};
		\node [style=none] (28) at (-7, 3.25) {};
		\node [style=none] (30) at (-7.25, 1) {$\dots$};
		\node [style=none] (33) at (-5, 1) {$\dots$};
		\node [style=bbox] (36) at (-10, 1) {};
		\node [style=bbox] (37) at (-4, 1) {};
		\node [style=none] (39) at (-6.5, 1.75) {};
		\node [style=none] (40) at (-6.5, 3.25) {};
		\node [style=rbox] (41) at (-8, 1) {};
		\node [style=rbox] (42) at (-6.5, 1) {};
		\node [style=bbwbox] (43) at (-7.25, 2.75) {};
		\node [style=none] (44) at (-9, 4) {{\footnotesize$i$-th}};
	\end{pgfonlayer}
	\begin{pgfonlayer}{edgelayer}
		\draw [in=-90, out=45] (18.center) to (20.center);
		\draw [in=-90, out=150, looseness=0.75] (17.center) to (19.center);
		\draw (26.center) to (28.center);
		\draw (25.center) to (27.center);
		\draw [bend left=90, looseness=1.25] (28.center) to (40.center);
		\draw (40.center) to (39.center);
	\end{pgfonlayer}
\end{tikzpicture}     \label{eqnmathbfh}
		\end{align}
		On the left, we have $p$  boxes 
		for $\varphi$ and $q$ blank boxes for $\psi$
		(dotted); the insertion is made in the $i$-th one. The total number of boxes on the right is $p+q$.

		Step~\ref{stepDi}
		is achieved with the following Lemma:
	\begin{lemma}\label{lemmah}
		With the  definition~\eqref{eqnhsigns} of $h$,
		 \begin{align} h\text{d}+\text{d}h=\gbos-\gamma_*^\bullet \ , \label{proofhomotopyh}\end{align}
		i.e.\
		$h$ is a homotopy from $\gamma_*^\bullet$ to $\gbos$.
		\end{lemma}
	
	\spaceplease
	\begin{proof} Let $\varphi$ and $\psi$ be cochains of degree $p$ and $q$, respectively. We will write $h_i$ instead of $h_i^{p,q}$ for better readability because the degree can be read off from the cochains that $h$ is being applied to.
		Thanks to
		\begin{align}
			h_0(\dif_0 \varphi\otimes \psi) &= h_0 \left( \  \begin{tikzpicture}
	\begin{pgfonlayer}{nodelayer}
		\node [style=flat box 2] (16) at (0, 0) {{$\varphi$}};
		\node [style=none] (17) at (-0.5, 0.25) {};
		\node [style=none] (18) at (0.5, 0.25) {};
		\node [style=none] (19) at (-1.25, 3) {};
		\node [style=none] (20) at (1.25, 3) {};
		\node [style=none] (22) at (0, -1) {$\dots$};
		\node [style=bbox] (23) at (-1, -1) {};
		\node [style=bbox] (24) at (1, -1) {};
		\node [style=bbox] (25) at (-1.25, 1.5) {};
	\end{pgfonlayer}
	\begin{pgfonlayer}{edgelayer}
		\draw [in=-90, out=45] (18.center) to (20.center);
		\draw [in=-90, out=150, looseness=0.75] (17.center) to (19.center);
	\end{pgfonlayer}
\end{tikzpicture} \otimes\begin{tikzpicture}
	\begin{pgfonlayer}{nodelayer}
		\node [style=flat box 2] (16) at (0, 0) {{$\psi$}};
		\node [style=none] (17) at (-0.5, 0.25) {};
		\node [style=none] (18) at (0.5, 0.25) {};
		\node [style=none] (19) at (-1, 2) {};
		\node [style=none] (20) at (1, 2) {};
		\node [style=none] (22) at (0, -1) {$\dots$};
		\node [style=rbox] (23) at (-1, -1) {};
		\node [style=rbox] (24) at (1, -1) {};
	\end{pgfonlayer}
	\begin{pgfonlayer}{edgelayer}
		\draw [in=-90, out=45] (18.center) to (20.center);
		\draw [in=-90, out=150, looseness=0.75] (17.center) to (19.center);
	\end{pgfonlayer}
\end{tikzpicture}   \     \right) \ =\    \begin{tikzpicture}
	\begin{pgfonlayer}{nodelayer}
		\node [style=flat box 2] (16) at (-7.25, 4.5) {{$\varphi$}};
		\node [style=none] (17) at (-7.75, 4.75) {};
		\node [style=none] (18) at (-6.75, 4.75) {};
		\node [style=none] (19) at (-8.5, 6) {};
		\node [style=none] (20) at (-5.75, 9) {};
		\node [style=none] (22) at (-8.75, 3) {$\dots$};
		\node [style=flat box 2] (24) at (-9.25, 6.75) {{$\psi$}};
		\node [style=none] (25) at (-9.5, 7) {};
		\node [style=none] (26) at (-9, 7) {};
		\node [style=none] (27) at (-9.5, 9) {};
		\node [style=none] (28) at (-9, 7.5) {};
		\node [style=none] (30) at (-12.25, 3) {$\dots$};
		\node [style=none] (33) at (-6, 3) {$\dots$};
		\node [style=bbox] (36) at (-10, 3) {};
		\node [style=bbox] (37) at (-5, 3) {};
		\node [style=none] (39) at (-8.5, 6) {};
		\node [style=none] (40) at (-8.5, 7.5) {};
		\node [style=rbox] (41) at (-13, 3) {};
		\node [style=rbox] (42) at (-11.5, 3) {};
		\node [style=bbwbox] (43) at (-9.25, 7) {};
		\node [style=bbox] (45) at (-7.25, 3) {};
	\end{pgfonlayer}
	\begin{pgfonlayer}{edgelayer}
		\draw [in=-90, out=45] (18.center) to (20.center);
		\draw [in=-90, out=150, looseness=0.75] (17.center) to (19.center);
		\draw (26.center) to (28.center);
		\draw (25.center) to (27.center);
		\draw [bend left=90, looseness=1.25] (28.center) to (40.center);
		\draw (40.center) to (39.center);
	\end{pgfonlayer}
\end{tikzpicture}\ = \gamma_*^\bullet(\psi\otimes\varphi)\ ,  \label{equationsforh1}   \\[1ex]
			h_p(\dif_{p+1}\varphi \otimes \psi) &= h_p \left( \  \begin{tikzpicture}
	\begin{pgfonlayer}{nodelayer}
		\node [style=flat box 2] (16) at (0, 0) {{$\varphi$}};
		\node [style=none] (17) at (-0.5, 0.25) {};
		\node [style=none] (18) at (0.5, 0.25) {};
		\node [style=none] (19) at (-1.25, 3) {};
		\node [style=none] (20) at (1.25, 3) {};
		\node [style=none] (22) at (0, -1) {$\dots$};
		\node [style=bbox] (23) at (-1, -1) {};
		\node [style=bbox] (24) at (1, -1) {};
		\node [style=bbox] (25) at (1.25, 1.75) {};
	\end{pgfonlayer}
	\begin{pgfonlayer}{edgelayer}
		\draw [in=-90, out=45] (18.center) to (20.center);
		\draw [in=-90, out=150, looseness=0.75] (17.center) to (19.center);
	\end{pgfonlayer}
\end{tikzpicture} \otimes\begin{tikzpicture}
	\begin{pgfonlayer}{nodelayer}
		\node [style=flat box 2] (16) at (0, 0) {{$\psi$}};
		\node [style=none] (17) at (-0.5, 0.25) {};
		\node [style=none] (18) at (0.5, 0.25) {};
		\node [style=none] (19) at (-1, 2) {};
		\node [style=none] (20) at (1, 2) {};
		\node [style=none] (22) at (0, -1) {$\dots$};
		\node [style=rbox] (23) at (-1, -1) {};
		\node [style=rbox] (24) at (1, -1) {};
	\end{pgfonlayer}
	\begin{pgfonlayer}{edgelayer}
		\draw [in=-90, out=45] (18.center) to (20.center);
		\draw [in=-90, out=150, looseness=0.75] (17.center) to (19.center);
	\end{pgfonlayer}
\end{tikzpicture}   \     \right)\ = \ \begin{tikzpicture}
	\begin{pgfonlayer}{nodelayer}
		\node [style=flat box 2] (16) at (-7.25, 4.5) {{$\varphi$}};
		\node [style=none] (17) at (-7.75, 4.75) {};
		\node [style=none] (18) at (-6.75, 4.75) {};
		\node [style=none] (20) at (-6, 6) {};
		\node [style=none] (22) at (-8.75, 3) {$\dots$};
		\node [style=none] (30) at (-2.75, 3) {$\dots$};
		\node [style=none] (33) at (-6, 3) {$\dots$};
		\node [style=bbox] (36) at (-10, 3) {};
		\node [style=bbox] (37) at (-5, 3) {};
		\node [style=rbox] (41) at (-3.5, 3) {};
		\node [style=rbox] (42) at (-2, 3) {};
		\node [style=bbox] (45) at (-7.25, 3) {};
		\node [style=none] (46) at (-6, 6) {};
		\node [style=flat box 2] (47) at (-5.25, 6.75) {{$\psi$}};
		\node [style=none] (48) at (-5.5, 7) {};
		\node [style=none] (49) at (-5, 7) {};
		\node [style=none] (50) at (-5.5, 7.5) {};
		\node [style=none] (51) at (-5, 8.75) {};
		\node [style=none] (52) at (-6, 6) {};
		\node [style=bbwbox] (54) at (-5.25, 7) {};
		\node [style=none] (55) at (-9, 9) {};
		\node [style=none] (56) at (-6, 6) {};
		\node [style=none] (57) at (-6, 7.5) {};
	\end{pgfonlayer}
	\begin{pgfonlayer}{edgelayer}
		\draw [bend right, looseness=1.25] (18.center) to (20.center);
		\draw (49.center) to (51.center);
		\draw (48.center) to (50.center);
		\draw [in=-90, out=180, looseness=0.75] (17.center) to (55.center);
		\draw [bend left=255, looseness=1.75] (50.center) to (57.center);
		\draw (57.center) to (56.center);
	\end{pgfonlayer}
\end{tikzpicture} \ = \gamma_*^\bullet(\varphi\otimes \psi)  \ ,  \label{equationsforh2}
			\end{align}
		we obtain 
		\begin{align}
			h(\dif \varphi \otimes \psi ) &= (-1)^{pq} h_0(\dif_0 \varphi\otimes \psi)+\sum_{i=1}^p (-1)^{i+(p-i)q} h_i (\dif_0 \varphi\otimes \psi)\\ & \phantom{=} +\underbrace{\sum_{i=0}^p \sum_{j=1}^p (-1)^{i+j+(p-i)q} h_i (\dif_j \varphi\otimes \psi)}_{=:S(\varphi,\psi)} + \sum_{i=0}^{p-1} (-1)^{i+p+1+(p-i)q} h_i (\dif_{p+1} \varphi \otimes \psi)\\ & \phantom{=}- h_p(\dif_{p+1}\varphi \otimes \psi) \\ &= (-1)^{pq} \gamma_*^\bullet(\psi \otimes \varphi)-\dif_0 h(\varphi\otimes \psi)+S(\varphi,\psi)\\ &\phantom{=}
			-(-1)^{p+q} \dif_{p+q} h(\varphi\otimes \psi)-\gamma_*^\bullet(\varphi\otimes\psi)\ . \label{hvarphipsi}
			\end{align}
		Next we further compute $S(\varphi,\psi)$: For $0\le i\le p$ and $1\le j\le p$, we find
		\begin{align}  h_i (\dif_j \varphi\otimes \psi) = \left\{ \begin{array}{cl} \dif_{j-1+q} h_i (\varphi\otimes \psi) \ , & i<j-1 \ , \\ h_i (\varphi\otimes \dif_{q+1} \psi) \ , & i=j-1 \ , \\ h_{i-1} (\varphi \otimes \dif_0\psi) \ , & i=j \ , \\  \dif_{j} h_{i-1} (\varphi\otimes \psi) \ , & i>j	 \ .   \end{array}\right.    \label{equationsforh3}
			\end{align}
		With the dummy variables	$\ell = j-1+q$
		and
		 $m=i-1$, this leads to
		\begin{align}
		S(\varphi,\psi)&=	- \sum_{\substack{0\le i\le p-1 \\ q\le \ell\le p-1+q \\ \ell > i+q}}	 (-1)^{i+\ell+(p-1-i)q} \dif_{\ell} h_i (\varphi\otimes \psi) - \sum_{ i=0}^{p-1} (-1)^{(p-i)q} h_i (\varphi \otimes \dif_{q+1} \psi) \\ &\phantom{=} +  \sum_{i=0}^{p-1} (-1)^{(p-1-i)q} h_i (\varphi\otimes \dif_0\psi) - \sum_{\substack{   0\le m\le p-1 \\ 1\le j\le p \\  m\ge j }}  (-1)^{m+j+(p-1-m)q}       \dif_{j} h_{m} (\varphi\otimes \psi) \\ &=K(\varphi,\psi)
	- (-1)^p \left(       h(\varphi \otimes (-1)^{q+1}\dif_{q+1} \psi)           +      h(\varphi \otimes \dif_0 \psi)\right)  \\ \small \text{with}\quad K(\varphi,\psi) &:= 	- 	\sum_{i=0}^{p-1}  \left(  \sum_{\substack{ q\le j\le p-1+q \\ j> i+q  }}    (-1)^{i+j+(p-1-i)q} \dif_{j} h_i (\varphi\otimes \psi)+
	\sum_{\substack{   1\le j\le p \\  i\ge j }}  (-1)^{i+j+(p-1-i)q} \dif_{j} h_{i} (\varphi\otimes \psi)\right) \ . 
		\end{align}
	\normalsize
As a consequence,	we obtain 
	\begin{align}
		S(\varphi,\psi)+(-1)^p h(\varphi\otimes \dif \psi)&=K(\varphi,\psi)+(-1)^p \sum_{i=0}^{p-1}\sum_{j=1}^q (-1)^{i+j+(p-1-i)(q+1)} h_i (\varphi \otimes \dif_j \psi) \\ &= K(\varphi,\psi)+(-1)^p \sum_{i=0}^{p-1}\sum_{j=1}^q (-1)^{i+j+(p-1-i)(q+1)} \dif_{i+j} h_i(\varphi \otimes \psi) \\ &\phantom{=} \left(\text{because} \    h_i (\varphi \otimes \dif_j \psi)=\dif_{i+j} h_i(\varphi \otimes \psi)\ \text{for}\ 0\le i\le p-1 \ , \ 1\le j\le q\right)\label{equationsforh4} \\
		&=
		K(\varphi,\psi)+(-1)^p \sum_{\substack{0\le i\le p-1 \\  i+1\le j\le i+q  }} (-1)^{j+(p-1-i)(q+1)} \dif_{j} h_i(\varphi \otimes \psi) \\ &=
		K(\varphi,\psi)- \sum_{\substack{0\le i\le p-1 \\  i+1\le j\le i+q  }} (-1)^{i+j+(p-1-i)q} \dif_{j} h_i(\varphi \otimes \psi) \\ &= - \sum_{i=0}^{p-1} \sum_{j=1}^{p+q-1} (-1)^{i+j+(p-1-i)q} \dif_{j} h_i(\varphi \otimes \psi) \\ &= - \dif h(\varphi \otimes \psi)  +\dif_0 h(\varphi\otimes \psi)+(-1)^{p+q}\dif_{p+q} h(\varphi\otimes \psi)   \ . \label{eqndifdif}
		\end{align}
	In summary,
	\begin{align}
	h(\dif(\varphi\otimes\psi))	&= h(\dif \varphi \otimes \psi )+(-1)^p h(\varphi\otimes \dif \psi) \\ &\stackrel{\eqref{hvarphipsi}}{=} (-1)^{pq} \gamma_*^\bullet(\psi \otimes \varphi)-\dif_0 h(\varphi\otimes \psi)+S(\varphi,\psi)\\ &\phantom{=}
	-(-1)^{p+q} \dif_{p+q} h(\varphi\otimes \psi)-\gamma_*^\bullet(\varphi\otimes\psi)+(-1)^p h(\varphi\otimes \dif \psi)\\ &\stackrel{\eqref{eqndifdif}}{=}  
	(-1)^{pq} \gamma_*^\bullet(\psi \otimes \varphi)-\gamma_*^\bullet(\varphi\otimes\psi)-\text{d}h(\varphi\otimes \psi) \ .         
		\end{align}
		This proves~\eqref{proofhomotopyh} and hence the Lemma.
		\end{proof}

\subparagraph{\normalfont \textsc{Step}~\ref{stepDii}.}
		We  now prepare ourselves to prove that the homotopy $h$ constructed algebraically in step~\ref{stepDi} agrees with
		the topologically extracted homotopy:
			The $E_2$-structure on $\cat{C}(I,\mathbb{A}^\bullet)$
		comes by construction from the braided $J_{\alg}$-algebra structure on an injective resolution $\alg^\bullet$ of $\alg \in Z(\cat{C})$.
		In this description, one needs
		$\mathbb{A}^\bullet = U\ialg$.
		We will cover afterwards the situation for an injective resolution of $\mathbb{A}\in\cat{C}$ which does not lift degree-wise to the Drinfeld center.
		The construction  will now be spelled out:
		
		\begin{lemma}\label{lemmaextracth}
			For any finite tensor category $\cat{C}$, let $\ialg$ be an injective resolution of the canonical algebra $\alg \in Z(\cat{C})$.
			Moreover, set $\mathbb{A}^\bullet := U\ialg$.
			\begin{pnum}
				
				\item Denote by $\Gamma :  \ialg\otimes\ialg \to \ialg$ the product of $\ialg$ (any extension of the product $\alg \otimes\alg \to \alg$ to the injective resolution) and by     $c_{\ialg,\ialg}:\ialg\otimes\ialg\to\ialg\otimes\ialg$ the braiding of the differential graded object $\ialg$ in $Z(\cat{C})$. 
				There is an essentially unique homotopy $\catf{H}:\Gamma \simeq \Gamma^\op := \Gamma \circ c_{\ialg,\ialg}$  (essentially unique means here up to higher homotopy)
				with the additional property that the precomposition with the coaugmentation $\iota_\alg^{\otimes 2} : \alg^{\otimes 2} \to \ialg^{\otimes 2}$
				\begin{align} Z(\cat{C})\left(\ialg^{\otimes 2},\ialg\right) \ra{\simeq}  Z\left(\cat{C})(\alg^{\otimes 2},\ialg\right) 
				\end{align} sends $\catf{H}$ to the zero self-homotopy of the map $\alg^{\otimes 2}\to\alg \ra{\iota_\alg} \ialg$ that first applies the product of $\alg$ (or the opposite product, which is equal) and then the coaugmentation. \label{step1lemmaex}

				\item 	If we apply the forgetful functor $U:Z(\cat{C})\to\cat{C}$, $\catf{H}$
				yields a  homotopy $U\catf{H}$  from the multiplication $\gamma^\bullet : \mathbb{A}^\bullet \otimes\mathbb{A}^\bullet \to\mathbb{A}^\bullet$ (extending the product $\gamma : \mathbb{A} \otimes\mathbb{A}\to\mathbb{A}$) to  \begin{align} \label{anothermulteqn} \bar \gamma^\bullet := \gamma^\bullet \circ U(c_{\ialg,\ialg})  : \mathbb{A}^\bullet \otimes\mathbb{A}^\bullet \ra{U(c_{\ialg,\ialg})} \mathbb{A}^\bullet \otimes \mathbb{A}^\bullet \ra{\gamma^\bullet}\mathbb{A}^\bullet \ . \end{align} 
				The multiplications $\gamma^\bullet$ and $\bar \gamma^\bullet$ on $\mathbb{A}^\bullet$ induce the multiplications $\gamma_*^\bullet$ and $\gbos$ on $\cat{C}(I,\mathbb{A}^\bullet)$, respectively. \label{lemmaextracth2}
				The homotopy $h$ from $\gamma_*^\bullet$
				to $\gbos$ extracted from the $E_2$-structure on $\cat{C}(I,\mathbb{A}^\bullet)$
				is the result of applying $\cat{C}(I,-)$ to $U\catf{H}$, i.e.\ it is given by the composition
				\begin{align}
			\label{eqngeth}		h:	\cat{C}(I,\mathbb{A}^\bullet) \otimes 	\cat{C}(I,\mathbb{A}^\bullet) \ra{\text{lax monoidal structure of $\cat{C}(I,-)$}}  \cat{C}(I,{\mathbb{A}^\bullet}^{\otimes 2})  \ra{\cat{C}(I,U\catf{H})} \cat{C}(I,\mathbb{A}^\bullet) \ . 
				\end{align}
				
				\end{pnum} 
			\end{lemma}

		\begin{proof}
			All of this is consequence of the definition of the braided operad $J_\alg$ (see~\eqref{eqncheckDdefined}),
			 a careful unpacking of the proofs of Proposition~\ref{propE2onhom} and Theorem~\ref{thmE2derivedhom},
			and the fact $\catf{H}$, as described in \ref{step1lemmaex}, is the result of the evaluation of the $J_\alg$-action 
			\begin{align}
				J_\alg(2) \to Z(\cat{C}) \left(  \ialg^{\otimes 2},\ialg\right)
			\end{align}
			on a 1-chain that is mapped
			by the epimorphism $J_\alg(2)\to C_*(E_2(2);k)$
			to the 1-chain in $E_2(2)$ described in Figure~\ref{figE22}. 
			\end{proof}
		
		One should be a bit careful to see \eqref{anothermulteqn} as an `opposite' multiplication because $U(c_{\ialg,\ialg})$
		is \emph{not} a part of a braiding in $\cat{C}$. Nonetheless,
		the images of the braiding in $Z(\cat{C})$ under $U$ turn $(\mathbb{A}^\bullet)^{\otimes n}$ into a $B_n$-module.

		If we want to use Lemma~\ref{lemmaextracth} to extract the topologically defined homotopy and compare it with the one concretely given in~\eqref{eqnhsigns}, there is a problem: 
		The resolution $\mathbb{A}^\bullet=\frint_{X\in\cat{C}} X\otimes X^\vee$
		used to write down the homotopy in~\eqref{eqnhsigns} does not lift degree-wise to $Z(\cat{C})$, i.e.\ it does not come with a half braiding.
		However, it comes with a structure that one could call a \emph{homotopy coherent half braiding}. This means that the half braiding for $\mathbb{A}$ (the non-crossing half braiding from~\eqref{defeqnhalfbraiding}) can be extended to $\mathbb{A}^\bullet$ (but without being a half braiding degree-wise). With this homotopy coherent half braiding, Lemma~\ref{lemmaextracth} remains in principle true, but a little more care is required in some places.   
		In order to provide the details,
		we will adapt the calculus for monoidal categories such that we can effectively compute with the resolution $\mathbb{A}^\bullet=\frint_{X\in\Proj\cat{C}} X\otimes X^\vee$.
		Recall that its $p$-cochains live in the product
		\begin{align}
			\prod_{X_0,\dots,X_p\in\Proj\cat{C}} \left(  X_0\otimes X_p^\vee  \right)^{  \cat{C}(X_p,X_{p-1})  \otimes \dots \otimes \cat{C}(X_1,X_0)    }  
		\end{align}  We will write the component of an $p$-cochain indexed by $(X_0,\dots,X_p)$ in the graphical calculus by
		\begin{align}\begin{tikzpicture}
	\begin{pgfonlayer}{nodelayer}
		\node [style=none] (0) at (-12, 7) {};
		\node [style=none] (1) at (-10, 7) {};
		\node [style=none] (3) at (-12, 3) {};
		\node [style=none] (4) at (-12, 2) {$\vdots$};
		\node [style=none] (5) at (-12, 1) {};
		\node [style=none] (7) at (-12, -3) {};
		\node [style=none] (8) at (-10, -3) {};
		\node [style=none] (9) at (-13, 6) {{\footnotesize $X_0$}};
		\node [style=none] (10) at (-9, 6) {{\footnotesize $X_p^\vee$}};
		\node [style=none] (11) at (-13, 4) {{\footnotesize $X_1$}};
		\node [style=none] (12) at (-13, 0) {{\footnotesize $X_{p-1}$}};
		\node [style=none] (13) at (-13, -2) {{\footnotesize $X_p$}};
		\node [style=bbox] (14) at (-12, 5) {};
		\node [style=bbox] (15) at (-12, -1) {};
	\end{pgfonlayer}
	\begin{pgfonlayer}{edgelayer}
		\draw [style=mydots, bend right=90, looseness=1.50] (7.center) to (8.center);
		\draw (8.center) to (1.center);
		\draw (7.center) to (15);
		\draw (15) to (5.center);
		\draw (3.center) to (14);
		\draw (14) to (0.center);
	\end{pgfonlayer}
\end{tikzpicture} \ .    \label{eqnmathbfA}
		\end{align}
		Formally speaking, this picture is to be read as the projection \begin{align} \mathbb{A}^p \to \left(  X_0\otimes X_p^\vee  \right)^{  \cat{C}(X_p,X_{p-1})  \otimes \dots \otimes \cat{C}(X_1,X_0)    } \ . \label{eqnpowering} \end{align} The boxes represent the vector spaces appearing in the exponent of the powering~\eqref{eqnpowering}.
		More precisely, the box
		between $X_{j+1}$ and $X_j$ for $0\le j\le p-1$ represents a blank argument that can be filled with a morphism $X_{j+1}\to X_j$. 
		The dotted line is purely mnemonic: It is not a coevaluation, but symbolizes the constraint that the object on the upper right (here: $X_p^\vee$) must be dual to the one in the left bottom (here: $X_p$).  With this notation, we can actually omit the labeling in \eqref{eqnmathbfA} because all components in the picture run over all labels, with the single constraint implemented through the dotted line.

		Now we can give the 
		\emph{homotopy coherent half braiding}
		of our resolution $\mathbb{A}^\bullet$ with $X\in\cat{C}$ by
		\begin{align}
			c_{\mathbb{A}^\bullet,X} \ := \ \begin{tikzpicture}
	\begin{pgfonlayer}{nodelayer}
		\node [style=none] (0) at (-12, 7) {};
		\node [style=none] (1) at (-10, 7) {};
		\node [style=none] (3) at (-12, 3) {};
		\node [style=none] (4) at (-12, 2) {$\vdots$};
		\node [style=none] (5) at (-12, 1) {};
		\node [style=none] (7) at (-12, -3) {};
		\node [style=none] (8) at (-10, -3) {};
		\node [style=none] (16) at (-9, 4) {};
		\node [style=none] (17) at (-9, -3) {};
		\node [style=none] (18) at (-13, -3) {};
		\node [style=none] (19) at (-13, 7) {};
		\node [style=none] (20) at (-13, 3) {};
		\node [style=none] (21) at (-13, 1) {};
		\node [style=none] (22) at (-7, 4) {};
		\node [style=none] (23) at (-7, -3) {};
		\node [style=bwbox] (52) at (-12.5, 5) {};
		\node [style=bwbox] (53) at (-12.5, -1) {};
		\node [style=wdisk] (54) at (-12, -1) {};
		\node [style=wdisk] (57) at (-12, 5) {};
		\node [style=none] (58) at (-7, -4) {$X$};
		\node [style=none] (59) at (-13, 8) {$X$};
		\node [style=none] (60) at (-8.25, 1) {$X^\vee$};
	\end{pgfonlayer}
	\begin{pgfonlayer}{edgelayer}
		\draw [style=mydots, bend right=90, looseness=1.50] (7.center) to (8.center);
		\draw (8.center) to (1.center);
		\draw (19.center) to (18.center);
		\draw [style=mydots, bend right=90, looseness=1.25] (18.center) to (17.center);
		\draw (17.center) to (16.center);
		\draw (23.center) to (22.center);
		\draw [bend left=90, looseness=1.50] (16.center) to (22.center);
		\draw (7.center) to (5.center);
		\draw (3.center) to (0.center);
	\end{pgfonlayer}
\end{tikzpicture}\ :\ \mathbb{A}^\bullet \otimes X \to X\otimes\mathbb{A}^\bullet \ . \label{eqnhalfbraidingcoh}
			\end{align}
		The lines drawn through the boxes indicate that  identities have been inserted. 
		Note that $c_{\mathbb{A}^\bullet,X}$ is determined up to a contractible choice by the fact that the restriction along the coaugmentation $\iota_\mathbb{A} : \mathbb{A}\to\mathbb{A}^\bullet$ gives us   $(X\otimes \iota_\mathbb{A})\circ c_{\mathbb{A},X}$, where $c_{\mathbb{A},X}$ is the
		usual non-crossing half braiding on $\mathbb{A}$.
			The following is a  consequence of this characterization of 	$c_{\mathbb{A}^\bullet,X}$:

		\begin{lemma}\label{lemmacohhalfbraiding}\begin{pnum}
				\item\label{lemmacohhalfbraiding1}
			The homotopy coherent half braiding \eqref{eqnhalfbraidingcoh} endows
			${\mathbb{A}^\bullet}^{\otimes n}$ with a homotopy coherent action of $B_n$ which is the essentially unique $B_n$-action making the $n$-fold coaugmentation
			\begin{align}	
					\mathbb{A}^{\otimes n} \ra{\simeq} \left(\mathbb{A}^\bullet\right)^{\otimes n}      \label{eqnequivBnmap}
				\end{align}
			$B_n$-equivariant up to coherent homotopy.
			
			\item \label{lemmacohhalfbraiding11}
			There is a unique homotopy $H:\gamma^\bullet \simeq \gamma^\bullet \circ c_{\mathbb{A}^\bullet,\mathbb{A}^\bullet}$ that becomes trivial if we precompose with the coaugmentation in the first slot.
			
			\item\label{lemmacohhalfbraiding2}
			The homotopy coherent half braiding  $	c_{\mathbb{A}^\bullet,-}$ is natural up to coherent homotopy: For objects $X$ and $Y$ in $\cat{C}$ (that can themselves be differential graded if $	c_{\mathbb{A}^\bullet,-}$ is understood degree-wise) 
			\begin{align}
				\begin{tikzpicture}
	\begin{pgfonlayer}{nodelayer}
		\node [style=flat box 2] (0) at (-6, 4.75) { \  \ $c_{\mathbb{A}^\bullet,Y}$ \ \ };
		\node [style=none] (1) at (-6.75, 4) {};
		\node [style=none] (2) at (-5.25, 4) {};
		\node [style=none] (3) at (-6.75, 6) {};
		\node [style=none] (4) at (-5.25, 6) {};
		\node [style=none] (5) at (-6.75, 2) {};
		\node [style=none] (6) at (-5.25, 2) {};
		\node [style=none] (9) at (-5.25, 6.75) {{\small $\mathbb{A}^\bullet$}};
		\node [style=none] (10) at (-6.75, 1.25) {{\small $\mathbb{A}^\bullet$}};
		\node [style=flat box 2] (11) at (-5.25, 3) {$?$};
		\node [style=none] (13) at (-5.25, 1.25) {{\small $X$}};
		\node [style=none] (14) at (-6.75, 6.75) {{\small $Y$}};
	\end{pgfonlayer}
	\begin{pgfonlayer}{edgelayer}
		\draw (5.center) to (1.center);
		\draw (1.center) to (3.center);
		\draw (2.center) to (4.center);
		\draw (6.center) to (2.center);
	\end{pgfonlayer}
\end{tikzpicture} \ \stackrel{N}{\simeq} \ (-1)^{\varepsilon} \begin{tikzpicture}
	\begin{pgfonlayer}{nodelayer}
		\node [style=flat box 2] (0) at (-6, 3) { \  \ $c_{\mathbb{A}^\bullet,X}$ \ \ };
		\node [style=none] (1) at (-6.75, 4) {};
		\node [style=none] (2) at (-5.25, 4) {};
		\node [style=none] (3) at (-6.75, 6) {};
		\node [style=none] (4) at (-5.25, 6) {};
		\node [style=none] (5) at (-6.75, 2) {};
		\node [style=none] (6) at (-5.25, 2) {};
		\node [style=none] (9) at (-5.25, 6.75) {{\small $\mathbb{A}^\bullet$}};
		\node [style=none] (10) at (-6.75, 1.25) {{\small $\mathbb{A}^\bullet$}};
		\node [style=flat box 2] (11) at (-6.75, 4.75) {$?$};
		\node [style=none] (13) at (-5.25, 1.25) {{\small $X$}};
		\node [style=none] (14) at (-6.75, 6.75) {{\small $Y$}};
	\end{pgfonlayer}
	\begin{pgfonlayer}{edgelayer}
		\draw (5.center) to (1.center);
		\draw (1.center) to (3.center);
		\draw (2.center) to (4.center);
		\draw (6.center) to (2.center);
	\end{pgfonlayer}
\end{tikzpicture}  :    \mathbb{A}^\bullet \otimes \cat{C}(X,Y) \otimes X \to Y \otimes \mathbb{A}^\bullet  \ , 
				\end{align}
			holds up to a coherent homotopy that we denote by $N$ (we suppress the dependence on $X$ and $Y$ in the notation).
			The box with the question mark can be filled with a morphism $X\to Y$. The integer $\varepsilon$ is the product of the degree of $?$ and the degree of $\mathbb{A}^\bullet$.
			This homotopy is the essentially unique one that becomes trivial if we precompose with the coaugmentation $\mathbb{A}\to\mathbb{A}^\bullet$.
			
			\item\label{lemmacohhalfbraiding3} Through the homotopies $N$,
			the maps
			\begin{align}
			\cat{C}(I,\mathbb{A}^\bullet)^{\otimes n} \to \cat{C}(I,{\mathbb{A}^\bullet}^{\otimes n})     \label{eqnequivmap}
		\end{align}
	become $B_n$-equivariant up to coherent homotopy, where on the left hand side the action is the strict action factoring through the permutation group.
			\end{pnum}
			\end{lemma}
		
		\begin{proof}
			The points~\ref{lemmacohhalfbraiding1}, \ref{lemmacohhalfbraiding11} and \ref{lemmacohhalfbraiding2} follow	 from the construction because the homotopy coherent half braiding reduces to the usual non-crossing braiding if we precompose with the coaugmentation. 
			One obtains~\ref{lemmacohhalfbraiding3} by specializing \ref{lemmacohhalfbraiding2} to $X=I$ and $Y=\mathbb{A}^\bullet$.
			\end{proof}

		We cannot only endow the homotopy end $\frint_{X\in\Proj\cat{C}} X\otimes X^\vee$
		with a homotopy coherent half braiding, but also  an injective resolution $U\ialg$ of $\mathbb{A} \in \cat{C}$ that comes from an arbitrary injective resolution of $\alg\in Z(\cat{C})$. In the latter case, we have of course an actual half braiding, so the statements of Lemma~\ref{lemmacohhalfbraiding} hold in a much stricter sense (for all points except~\ref{lemmacohhalfbraiding2}, the coherence data are trivial). 
		Now we make two observations:
		\begin{itemize}
			\item The injective resolutions $\frint_{X\in\Proj\cat{C}} X\otimes X^\vee$ and $U\ialg$ are homotopy equivalent, and their homotopy coherent half braidings agree, possibly up to higher homotopy. The first part is standard, the second part comes from the fact that the half braiding is by construction determined by its restriction to $\mathbb{A}$.
			In fact, \emph{all} injective resolutions of $\mathbb{A}$ with homotopy coherent half braiding are equivalent.
			
			\item The homotopy $h:\gamma_*^\bullet \simeq \gbos$ described in Lemma~\ref{lemmaextracth} depends only on $U\ialg$ as object with half braiding. Indeed, we get the multiplication $\gamma^\bullet$ from the multiplication on $\mathbb{A}$, $\bar \gamma^\bullet$ from the braid group action~\eqref{lemmaextracth} which is a special case of the half braiding, and the needed homotopy from $\gamma^\bullet$ to $\bar \gamma^\bullet$ from Lemma~\ref{lemmaextracth} \ref{lemmaextracth2} can once again be characterized by the fact that it becomes trivial when precomposed with the coaugmentation. With these ingredients, we can obtain $h$ via \eqref{eqngeth}.
			
			\end{itemize}
		These two observations imply: We can compute $h$ from \emph{any} injective resolution of $\mathbb{A}$ equipped with homotopy coherent half braiding. This will possibly change $h$ by a higher homotopy, but we are only interested in $h$ up to higher homotopy anyway.

		\begin{lemma}\label{lemmahHN}
		For the resolution $\frint_{X\in\Proj\cat{C}} X\otimes X^\vee$ of $\mathbb{A}$ and the half braiding 
		\eqref{eqnhalfbraidingcoh}, the topologically extracted homotopy $h:\gamma^\bullet \simeq \gbos$ is given by the homotopy
		\begin{align}
		\label{eqnhfinal}	h \ : \ \gamma^\bullet_* \ = \ \begin{tikzpicture}
	\begin{pgfonlayer}{nodelayer}
		\node [style=flat box 2] (0) at (-12, 6.25) {{\small $\gamma^\bullet$}};
		\node [style=flat box 2] (1) at (-12.75, 3) {{\small $?_1$}};
		\node [style=flat box 2] (2) at (-11.25, 3) {{\small $?_2$}};
		\node [style=none] (3) at (-12, 9) {};
		\node [style=none] (4) at (-12, 9.5) {{\small $\mathbb{A}^\bullet$}};
		\node [style=flat box 2] (5) at (-3, 4.75) { \  \ $c_{\mathbb{A}^\bullet,\mathbb{A}^\bullet}$ \ \ };
		\node [style=none] (6) at (-3.75, 1) {};
		\node [style=none] (7) at (-2.25, 3) {};
		\node [style=none] (8) at (-3.75, 6) {};
		\node [style=none] (9) at (-2.25, 6) {};
		\node [style=flat box 2] (10) at (-3.75, 1) {{\small $?_1$}};
		\node [style=flat box 2] (11) at (-2.25, 3) {{\small $?_2$}};
		\node [style=flat box 2] (12) at (-3, 7.25) {{\small $\gamma^\bullet$}};
		\node [style=none] (13) at (-3, 9) {};
		\node [style=none] (14) at (-3, 9.5) {{\small $\mathbb{A}^\bullet$}};
		\node [style=none] (15) at (-9, 7) {};
		\node [style=none] (16) at (-6, 7) {};
		\node [style=none] (17) at (-7.5, 7.5) {$H$};
		\node [style=none] (18) at (-7.5, 6.5) {};
		\node [style=none] (19) at (-7.5, 6.5) {{\footnotesize Lemma~\ref{lemmacohhalfbraiding} \ref{lemmacohhalfbraiding11}}};
		\node [style=none] (20) at (-14, 10) {};
		\node [style=none] (21) at (-10, 10) {};
		\node [style=none] (22) at (-10, 4) {};
		\node [style=none] (23) at (-14, 4) {};
		\node [style=none] (24) at (-5, 10) {};
		\node [style=none] (25) at (-1, 10) {};
		\node [style=none] (26) at (-5, 4) {};
		\node [style=none] (27) at (-1, 4) {};
		\node [style=none] (28) at (-4.75, 6) {};
		\node [style=none] (29) at (-1.25, 6) {};
		\node [style=none] (30) at (-4.75, 2) {};
		\node [style=none] (31) at (-1.25, 2) {};
		\node [style=none] (32) at (0, 4) {};
		\node [style=none] (33) at (3, 4) {};
		\node [style=none] (34) at (1.5, 4.5) {$N$};
		\node [style=none] (35) at (1.5, 3.5) {};
		\node [style=none] (36) at (1.5, 3.5) {{\footnotesize Lemma~\ref{lemmacohhalfbraiding} \ref{lemmacohhalfbraiding2}}};
		\node [style=none] (39) at (6.5, 1.25) {};
		\node [style=none] (40) at (5, 6) {};
		\node [style=none] (41) at (6.5, 6) {};
		\node [style=flat box 2] (42) at (6.5, 1.25) {{\small $?_1$}};
		\node [style=flat box 2] (43) at (5, 4) {{\small $?_2$}};
		\node [style=flat box 2] (44) at (5.75, 7.25) {{\small $\gamma^\bullet$}};
		\node [style=none] (45) at (5.75, 9) {};
		\node [style=none] (46) at (5.75, 9.5) {{\small $\mathbb{A}^\bullet$}};
		\node [style=none] (51) at (4, 6) {};
		\node [style=none] (52) at (7.5, 6) {};
		\node [style=none] (53) at (4, 2) {};
		\node [style=none] (54) at (7.5, 2) {};
		\node [style=none] (55) at (2.75, 7.25) {{\footnotesize $(-1)^{|?_1||?_2|}$}};
	\end{pgfonlayer}
	\begin{pgfonlayer}{edgelayer}
		\draw (1) to (0);
		\draw (2) to (0);
		\draw (0) to (3.center);
		\draw (6.center) to (8.center);
		\draw (7.center) to (9.center);
		\draw (12) to (13.center);
		\draw (8.center) to (12);
		\draw (9.center) to (12);
		\draw [style=end arrow] (15.center) to (16.center);
		\draw [style=bluedashed] (20.center) to (21.center);
		\draw [style=bluedashed] (21.center) to (22.center);
		\draw [style=bluedashed] (23.center) to (22.center);
		\draw [style=bluedashed] (23.center) to (20.center);
		\draw [style=bluedashed] (26.center) to (27.center);
		\draw [style=bluedashed] (27.center) to (25.center);
		\draw [style=bluedashed] (25.center) to (24.center);
		\draw [style=bluedashed] (24.center) to (26.center);
		\draw [style=REDdotted] (28.center) to (29.center);
		\draw [style=REDdotted] (29.center) to (31.center);
		\draw [style=REDdotted] (31.center) to (30.center);
		\draw [style=REDdotted] (30.center) to (28.center);
		\draw [style=end arrow] (32.center) to (33.center);
		\draw (39.center) to (41.center);
		\draw (44) to (45.center);
		\draw (40.center) to (44);
		\draw (41.center) to (44);
		\draw [style=REDdotted] (51.center) to (52.center);
		\draw [style=REDdotted] (52.center) to (54.center);
		\draw [style=REDdotted] (54.center) to (53.center);
		\draw [style=REDdotted] (53.center) to (51.center);
		\draw (40.center) to (43);
	\end{pgfonlayer}
\end{tikzpicture} \ = \ \gbos \, , 
			\end{align}
		where the  frames indicate the areas that $H$ and $N$ are applied to. 
		\end{lemma}
	
	\begin{proof}
		For \emph{any} injective resolution  $R=(\mathbb{A}^\bullet,c)$ of $\mathbb{A}$ with homotopy coherent half braiding, one obtains the homotopies $H$ and $N$ that we can use to associate to $R$ the homotopy $h_R$ via the formula~\eqref{eqnhfinal}.
		But all of such injective resolutions with homotopy coherent half braidings are equivalent as explained, and hence so are the $h_R$.
		Now it remains to verify that $h_R$ with $R=U\ialg$ reduces to the homotopy from \eqref{eqngeth} from Lemma~\ref{lemmaextracth} \ref{lemmaextracth2}. But this is true because the homotopy $N$, in this case, happens to be trivial since we have an \emph{actual} half braiding.
		\end{proof}
	
	With the following Lemma,
	we complete step~\ref{stepDii}:
	
	\begin{lemma}\label{lemmahish}
		The homotopy $h: \gamma^\bullet_*\simeq \gbos$ from Lemma~\ref{lemmah} given for the injective resolution
		$\mathbb{A}^\bullet = \frint_{X\in\Proj\cat{C}} X\otimes X^\vee$
		agrees up to higher homotopy 
		with the topologically extracted homotopy $\gamma_*^\bullet \simeq \gbos$
		 of the $E_2$-algebra $\cat{C}(I,\mathbb{A}^\bullet)$.
	\end{lemma}

	\begin{proof}
		For the entire proof, we fix the injective resolution
		$\mathbb{A}^\bullet = \frint_{X\in\Proj\cat{C}} X\otimes X^\vee$.
		Rephrasing Lemma~\ref{lemmahHN}, the topologically extracted homotopy
		$h$ is obtained by applying $\cat{C}(I,-)$
		to the homotopy $L$ of maps $\xi,\bar \xi: \mathbb{A}^\bullet \otimes \cat{C}(I,\mathbb{A}^\bullet) \to \mathbb{A}^\bullet$
		defined by
		\begin{align}
			L \ :\ \quad  \xi  \ \ :=\
			\begin{tikzpicture}
	\begin{pgfonlayer}{nodelayer}
		\node [style=flat box 2] (0) at (-12, 6.25) {{\small $\gamma^\bullet$}};
		\node [style=flat box 2] (2) at (-11.25, 3) {{\small $?_2$}};
		\node [style=none] (3) at (-12, 9) {};
		\node [style=none] (4) at (-12, 9.5) {{\small $\mathbb{A}^\bullet$}};
		\node [style=flat box 2] (5) at (-3, 4.75) { \  \ $c_{\mathbb{A}^\bullet,\mathbb{A}^\bullet}$ \ \ };
		\node [style=none] (6) at (-3.75, 1) {};
		\node [style=none] (7) at (-2.25, 3) {};
		\node [style=none] (8) at (-3.75, 6) {};
		\node [style=none] (9) at (-2.25, 6) {};
		\node [style=flat box 2] (11) at (-2.25, 3) {{\small $?_2$}};
		\node [style=flat box 2] (12) at (-3, 7.25) {{\small $\gamma^\bullet$}};
		\node [style=none] (13) at (-3, 9) {};
		\node [style=none] (14) at (-3, 9.5) {{\small $\mathbb{A}^\bullet$}};
		\node [style=none] (15) at (-9, 7) {};
		\node [style=none] (16) at (-6, 7) {};
		\node [style=none] (17) at (-7.5, 7.5) {$H$};
		\node [style=none] (18) at (-7.5, 6.5) {};
		\node [style=none] (19) at (-7.5, 6.5) {{\footnotesize Lemma~\ref{lemmacohhalfbraiding} \ref{lemmacohhalfbraiding11}}};
		\node [style=none] (20) at (-14, 10) {};
		\node [style=none] (21) at (-10, 10) {};
		\node [style=none] (22) at (-10, 4) {};
		\node [style=none] (23) at (-14, 4) {};
		\node [style=none] (24) at (-5, 10) {};
		\node [style=none] (25) at (-1, 10) {};
		\node [style=none] (26) at (-5, 4) {};
		\node [style=none] (27) at (-1, 4) {};
		\node [style=none] (28) at (-4.75, 6) {};
		\node [style=none] (29) at (-1.25, 6) {};
		\node [style=none] (30) at (-4.75, 2) {};
		\node [style=none] (31) at (-1.25, 2) {};
		\node [style=none] (32) at (0, 4) {};
		\node [style=none] (33) at (3, 4) {};
		\node [style=none] (34) at (1.5, 4.5) {$N$};
		\node [style=none] (35) at (1.5, 3.5) {};
		\node [style=none] (36) at (1.5, 3.5) {{\footnotesize Lemma~\ref{lemmacohhalfbraiding} \ref{lemmacohhalfbraiding2}}};
		\node [style=none] (39) at (6.75, 1.25) {};
		\node [style=none] (40) at (5.25, 6) {};
		\node [style=none] (41) at (6.75, 6) {};
		\node [style=flat box 2] (43) at (5.25, 4) {{\small $?_2$}};
		\node [style=flat box 2] (44) at (6, 7.25) {{\small $\gamma^\bullet$}};
		\node [style=none] (45) at (6, 9) {};
		\node [style=none] (46) at (6, 9.5) {{\small $\mathbb{A}^\bullet$}};
		\node [style=none] (51) at (4.25, 6) {};
		\node [style=none] (52) at (7.75, 6) {};
		\node [style=none] (53) at (4.25, 2) {};
		\node [style=none] (54) at (7.75, 2) {};
		\node [style=none] (55) at (-13, 3) {};
		\node [style=none] (56) at (3.25, 7.25) {{\footnotesize $(-1)^{\varepsilon}$}};
		\node [style=none] (57) at (6.75, 0.5) {{\small $\mathbb{A}^\bullet$}};
		\node [style=none] (58) at (-13, 2.25) {{\small $\mathbb{A}^\bullet$}};
		\node [style=none] (59) at (-3.75, 0) {{\small $\mathbb{A}^\bullet$}};
	\end{pgfonlayer}
	\begin{pgfonlayer}{edgelayer}
		\draw (2) to (0);
		\draw (0) to (3.center);
		\draw (6.center) to (8.center);
		\draw (7.center) to (9.center);
		\draw (12) to (13.center);
		\draw (8.center) to (12);
		\draw (9.center) to (12);
		\draw [style=end arrow] (15.center) to (16.center);
		\draw [style=bluedashed] (20.center) to (21.center);
		\draw [style=bluedashed] (21.center) to (22.center);
		\draw [style=bluedashed] (23.center) to (22.center);
		\draw [style=bluedashed] (23.center) to (20.center);
		\draw [style=bluedashed] (26.center) to (27.center);
		\draw [style=bluedashed] (27.center) to (25.center);
		\draw [style=bluedashed] (25.center) to (24.center);
		\draw [style=bluedashed] (24.center) to (26.center);
		\draw [style=REDdotted] (28.center) to (29.center);
		\draw [style=REDdotted] (29.center) to (31.center);
		\draw [style=REDdotted] (31.center) to (30.center);
		\draw [style=REDdotted] (30.center) to (28.center);
		\draw [style=end arrow] (32.center) to (33.center);
		\draw (39.center) to (41.center);
		\draw (44) to (45.center);
		\draw (40.center) to (44);
		\draw (41.center) to (44);
		\draw [style=REDdotted] (51.center) to (52.center);
		\draw [style=REDdotted] (52.center) to (54.center);
		\draw [style=REDdotted] (54.center) to (53.center);
		\draw [style=REDdotted] (53.center) to (51.center);
		\draw (40.center) to (43);
		\draw (0) to (55.center);
	\end{pgfonlayer}
\end{tikzpicture} \ =: \ \bar \xi \ . 
		\end{align}
		In short,
		\begin{align} h=\cat{C}(I,L)\ . \label{eqnhandL}
			\end{align}
		If we precompose with the coaugmentation,
		$L$ becomes trivial (this follows for $H$ from Lemma~\ref{lemmacohhalfbraiding} \ref{lemmacohhalfbraiding11} and for $N$ by from Lemma~\ref{lemmacohhalfbraiding} \ref{lemmacohhalfbraiding2}). 
		But
		\begin{align}
			\cat{C}(\mathbb{A}^\bullet \otimes \cat{C}(I,\mathbb{A}^\bullet), \mathbb{A}^\bullet)   \ra{  (\iota_\mathbb{A} \otimes \id_{ \cat{C}(I,\mathbb{A}^\bullet) }  )^*   }
			\cat{C}(\mathbb{A} \otimes \cat{C}(I,\mathbb{A}^\bullet), \mathbb{A}^\bullet) \ . 
		\end{align}
	is again a trivial fibration which implies that $L$ is the essentially unique homotopy $\xi\simeq \bar \xi$ that becomes trivial when precomposed with the coaugmentation. 
		This allows us to give a model for $L$:
		First we define $L_i^{p,q}:\mathbb{A}^p \otimes \cat{C}(I,\mathbb{A}^q)\to\mathbb{A}^{p+q-1}$
		through
		\begin{align}L_i^{p,q}(-\otimes\psi) = \begin{tikzpicture}
	\begin{pgfonlayer}{nodelayer}
		\node [style=none] (1) at (-11, 4) {};
		\node [style=none] (4) at (-13, 2) {$\vdots$};
		\node [style=none] (8) at (-11, -6) {};
		\node [style=none] (22) at (-13, -2) {};
		\node [style=none] (25) at (-14, 8) {$\vdots$};
		\node [style=none] (26) at (-14, 11) {};
		\node [style=none] (28) at (-14, 9) {};
		\node [style=none] (34) at (-13, 0.5) {};
		\node [style=none] (35) at (-14, 7) {};
		\node [style=none] (37) at (-14, 6.25) {};
		\node [style=none] (40) at (-11, 11) {};
		\node [style=bbox] (52) at (-14, 10) {};
		\node [style=none] (55) at (-16, 6) {{\footnotesize $i$-th}};
		\node [style=flat box 2] (56) at (-13.75, 5) {{$\psi$}};
		\node [style=none] (57) at (-14, 5.25) {};
		\node [style=none] (58) at (-13.5, 5.25) {};
		\node [style=none] (59) at (-14, 6.25) {};
		\node [style=none] (60) at (-13.5, 5.75) {};
		\node [style=none] (62) at (-13, 4.25) {};
		\node [style=none] (63) at (-13, 5.75) {};
		\node [style=rbox] (64) at (-13, 0.5) {};
		\node [style=rbox] (65) at (-13, 3) {};
		\node [style=bbwbox] (66) at (-13.75, 5.25) {};
		\node [style=bbox] (67) at (-13, -1) {};
		\node [style=none] (69) at (-13, 4.25) {};
		\node [style=none] (71) at (-13, -3) {$\vdots$};
		\node [style=none] (72) at (-13, -4) {};
		\node [style=none] (73) at (-13, -6) {};
		\node [style=bbox] (74) at (-13, -5) {};
	\end{pgfonlayer}
	\begin{pgfonlayer}{edgelayer}
		\draw (8.center) to (1.center);
		\draw [in=270, out=90] (22.center) to (34.center);
		\draw (37.center) to (35.center);
		\draw (1.center) to (40.center);
		\draw (28.center) to (52);
		\draw (52) to (26.center);
		\draw (58.center) to (60.center);
		\draw (57.center) to (59.center);
		\draw [bend left=90, looseness=1.25] (60.center) to (63.center);
		\draw (63.center) to (62.center);
		\draw (73.center) to (72.center);
		\draw (69.center) to (65);
		\draw [style=mydots, bend right=90, looseness=1.50] (73.center) to (8.center);
	\end{pgfonlayer}
\end{tikzpicture}  \quad \text{for}\quad \psi\in\cat{C}(I,\mathbb{A}^q) \ .    \label{eqnmathbfA2}
		\end{align}
		The operations $L_i^{p,q}$ are binary, and
		the boxes and dotted boxes
		are associated to the first and second argument, respectively.
		Next we set \begin{align}
			\label{defeqnL}L^{p,q} := \sum_{i=0}^{p-1} (-1)^{i+(p-1-i)q} L_i^{p,q} : \mathbb{A}^p \otimes \cat{C}(I,\mathbb{A}^q)\to\mathbb{A}^{p+q-1}\ . \end{align}
		Then 
		\begin{align} L\text{d}+\text{d}L=\bar\xi-\xi \ , \end{align}
		i.e.\
		$L$ is a homotopy from $\xi$ to $\bar\xi$.
			This can be confirmed with essentially the same computation	as for Lemma~\ref{lemmah}.
		In order to verify that this homotopy really models $L$, we need to verify that it vanishes when precomposed with the coaugmentation.
		But this follows from $L^{0,q}=0$.

		Now we can use the model~\eqref{defeqnL} to compute $h$ via \eqref{eqnhandL}.
		This gives us exactly the formula for $h$ in Lemma~\ref{lemmah}.
		\end{proof}

\subparagraph{\normalfont \textsc{Step}~\ref{stepDiii}.}	
We can now finally compute algebraically the Gerstenhaber bracket of the $E_2$-algebra $\cat{C}(I,\mathbb{A}^\bullet)$.
So far, we obtained
for the $E_2$-algebra $\cat{C}(I,\mathbb{A}^\bullet)$
the homotopy between multiplication and opposite multiplication coming from the path in $E_2(2)$ given in Figure~\ref{figE22}.
The key technical ingredient for this is Lemma~\ref{lemmahish} that tells us that the homotopy $h$ concretely defined via~\eqref{eqnhsigns} gives us a model for this homotopy.

Finally, in this step, we want to compute the Gerstenhaber bracket for the $E_2$-algebra
$\cat{C}(I,\mathbb{A}^\bullet)$.
To this end, we compute the
binary degree one operation $b$ corresponding to 
the fundamental class of
$E_2(2)\simeq \mathbb{S}^1$ (the orientation 
comes here from preferring the braiding over its inverse). From this operation $b$ and an additional sign, we obtain the Gerstenhaber bracket as we will explain in a moment.

We can obtain the loop in $E_2(2)\simeq \mathbb{S}^1$
corresponding to the fundamental class by composing two half circular paths.
We have established that the homotopy $h$ is the evaluation of the $E_2$-algebra $\cat{C}(I,\mathbb{A}^\bullet)$ on the first of these half circles, at least up to higher homotopy.
As a consequence,
 the binary degree one operation $b$
is the composition $ h+ h\tau$ of the homotopy $h$ from $\gamma_*^\bullet$
to $\gbos$ and
 the homotopy $ h \tau$ from
 $\gbos$ to $\gamma_*^\bullet$,
where $\tau$ is the symmetric braiding in $\Ch$:
\begin{align}
	b (\varphi,\psi) = h(\varphi,\psi) + (-1)^{pq} h(\psi,\varphi) \quad \text{for}\quad \varphi \in \cat{C}(I,\mathbb{A}^p)\ , \ \psi \in \cat{C}(I,\mathbb{A}^q) \ .    \label{eqnbformula}    
	\end{align}
 The connection to the Gerstenhaber bracket is $[\varphi,\psi]=(-1)^{p}b(\varphi,\psi)$ \cite[page~220]{salvatorewahl}
 (this additional sign ensures the anti-symmetry of the Gerstenhaber bracket),
 which leads us to
 \begin{align}
 	[\varphi,\psi]=  (-1)^p h(\varphi,\psi) + (-1)^{pq+p} h(\psi,\varphi) \ .    \label{eqnbrcketonhom}
 	\end{align}
 
 Under order to express compactly
 the Gerstenhaber bracket induced on Hochschild cohomology 
  under the equivalence $\rint_{X\in\Proj \cat{C}}    \cat{C}(X,X)\simeq \cat{C}(I,\mathbb{A}^\bullet)$, 
  recall 
  the $i$-th partial composition operation
  \begin{align} \circ_i : \left(\rint_{X\in\Proj \cat{C}}    \cat{C}(X,X)\right)^p\otimes  \left(\rint_{X\in\Proj \cat{C}}    \cat{C}(X,X)\right)^q&\to  \left(\rint_{X\in\Proj \cat{C}}    \cat{C}(X,X)\right)^{p+q-1} \ , \quad 0\le i\le p-1 \\ 
  	\alpha\otimes\beta & \mapsto \alpha \circ_i \beta \ , 
  \end{align}
  where 
  \begin{align}
  	(\alpha \circ_i \beta)_{X_0,\dots,X_{i-1},Y_0,\dots,Y_q,X_{i+1},\dots,X_p} := \alpha_{X_0,\dots,X_{i-1},Y_0,Y_q,X_{i+1},\dots,X_p} (-, \beta_{Y_0,\dots,Y_q} , - ) \ . 
  \end{align}
The operations $\circ_i$ are used to define the circle product in the sense of \cite[Definition~1.4.1]{witherspoon}
\begin{align}
	\alpha \circ \beta := \sum_{i=0}^{p-1} (-1)^{(q-1)i} \alpha \circ_i \beta \ , \quad |\alpha|=p\ , \quad  |\beta|=q \ . 
\end{align}
  
  \begin{lemma}\label{lemmagerstenhaberbracket}
  	Under the equivalence, $\rint_{X\in\Proj \cat{C}}    \cat{C}(X,X)\simeq \cat{C}(I,\mathbb{A}^\bullet)$
  	the Gerstenhaber bracket of $\cat{C}(I,\mathbb{A}^\bullet)$ translates to
  	the bracket
  	\begin{align}
  		[\alpha,\beta]  = - (-1)^{(p-1)(q-1)} \alpha \circ \beta+\beta \circ \alpha  \ , \label{eqngerstenhaberbracket}
  	\end{align}
  where $\alpha$ and $\beta$ are in degree $p$ and $q$, respectively.
  \end{lemma}

With our sign conventions, \eqref{eqngerstenhaberbracket}
is the `standard' Gerstenhaber bracket on Hochschild cohomology (we comment on other conventions in Remark~\ref{remconv}).
This finishes the proof that the $E_2$-algebra $\bspace{\cat{C}(I,\mathbb{A}^\bullet)}{\gamma_*^\bullet}$ is a (very explicit) solution to Deligne's Conjecture.

 \begin{proof}[\slshape Proof of Lemma~\ref{lemmagerstenhaberbracket}]
 	Under the equivalence 	$\rint_{X\in\Proj \cat{C}}    \cat{C}(X,X)\simeq \cat{C}(I,\mathbb{A}^\bullet)$,
  the part $h_i^{p,q}$ of the homotopy $h$ from step~\ref{stepDi}, \eqref{eqnmathbfh} (here $p,q\ge 0$ and $0\le i\le p-1$) translates to the partial composition operation $\circ_i$.
Hence, the
homotopy $h$  on $\cat{C}(I,\mathbb{A}^\bullet)$ from \eqref{eqnhsigns} translates to the degree one operation on $\rint_{X\in\Proj \cat{C}}    \cat{C}(X,X)$ sending $\alpha$ in degree $p$ and $\beta$ in degree $q$
to
\begin{align}
	\sum_{i=0}^{n-1} (-1)^{i+(p-1-i)q} \alpha \circ_i \beta = (-1)^{pq+q} \alpha \circ \beta  \ . \label{eqncircop}
	\end{align}
But then the
  bracket~\eqref{eqnbrcketonhom} translates to the bracket
\begin{align}
	[\alpha,\beta] &= 
	(-1)^{p+pq+q} \alpha \circ \beta +\beta \circ \alpha = - (-1)^{(p-1)(q-1)} \alpha \circ \beta+\beta \circ \alpha  \ . 
	\end{align}
\end{proof}

\begin{remark}\label{remconv}
	In many places in the literature, including the textbook \cite{witherspoon},
	a different convention for the cup product is used.
	This alternative cup product $\smile'$ relates to ours by $\alpha \smile' \beta = (-1)^{pq} \alpha \smile \beta$ with $p=|\alpha|$ and $q=|\beta|$.  
	For us, this convention would be a bad choice because it does not turn $\smile'$ into a chain map (only up to a sign), but it can be convenient for other purposes. If we want to obtain the bracket associated to $\smile'$ rather than $\smile$, we need to multiply the homotopy $h$ above also degree-wise with $(-1)^{pq}$. This entails that $b$ on $\cat{C}(I,\mathbb{A}^\bullet)$ must also be multiplied with $(-1)^{pq}$, thereby giving us $b'$. But
	\begin{align}
		b' (\varphi,\psi) = (-1)^{pq} b(\varphi,\psi) = b(\psi,\varphi) \quad \text{for}\quad \varphi \in \cat{C}(I,\mathbb{A}^p)\ , \ \psi \in \cat{C}(I,\mathbb{A}^q)     \label{eqnbformula2}    
	\end{align}
because $b$ is graded commutative by definition. 
This changes the bracket that we obtain on $\rint_{X\in\Proj \cat{C}}    \cat{C}(X,X)$
to
\begin{align}
	[\alpha,\beta]' &=   (-1)^{p+pq+q} \beta\circ \alpha+\alpha\circ \beta = \alpha \circ \beta - (-1)^{(p-1)(q-1)} \beta \circ \alpha \ . 
\end{align}
This agrees now with \cite[Definition~1.4.1]{witherspoon}.
	\end{remark}

\begin{remark}\label{remsharp}
	By considering for $\cat{C}$ the modules over the group algebra of a finite group,
	we deduce from Theorem~\ref{thhmcomparisondeligne} and the computations in
\cite{lezhou}
that the Gerstenhaber bracket
of the $E_2$-algebras constructed from the homotopy invariants of a braided commutative algebra is generally non-trivial.
\end{remark}

\begin{example}[Quantum groups]\label{exquantumgroups}
	Let $H$ be a finite-dimensional Hopf algebra with antipode $S:H\to H$.
	 Then the category $\cat{C}$ of a finite-dimensional $H$-modules is a finite tensor category and its canonical end $\mathbb{A}$ is isomorphic to the adjoint representation $H_\text{ad}$ \cite[Theorem~7.4.13]{kl}, i.e.\ $H$ with action $x.y = x' y S(x'')$ for $x,y \in H$, where $\Delta x = x' \otimes x''$ is the Sweedler notation for the coproduct.
	 With~\eqref{eqnAHH}, we can now write the Hochschild cohomology of $H$ as $\Ext_H^*(k,H_\text{ad})$, see \cite[Section~2.2]{bichon}.  
	 The multiplication of $H$ is ad-equivariant and therefore endows $H_\text{ad}$ with structure of an algebra in $H$-modules. This gives us the multiplicative structure on Hochschild cohomology. If $H$ is a small quantum group $\catf{u}_q(\mathfrak{sl}_2)$ at a primitive root of unity, the entire multiplicative structure on the Hochschild cohomology is known  as graded ring \cite{lq}. What Theorem~\ref{thhmcomparisondeligne} tells us in the case of categories of $H$-modules is how the Gerstenhaber structure on $HH^*(H)$ is determined by the canonical structure of $H_\text{ad}$ as an object in the Drinfeld center of the category of $H$-modules, i.e.\ its structure as a module over the quantum double $D(H)$ and its structure as a braided commutative algebra in $D(H)$-modules. This is of course
	 just the informal summary; the actual construction is the abstract one from Theorem~\ref{thmE2derivedhom} --- and as we have seen, the actual proof that this produces the `usual' Gerstenhaber bracket is quite involved (and the main point of this article). 
	 Let us emphasize again that the motivation behind writing 
	 the Gerstenhaber structure on Hochschild cohomology (or rather the $E_2$-structure on the Hochschild cochains) is (at least for now) \emph{not} the computation of Gerstenhaber brackets in examples. It is rather about condensing all of the complicated differential graded information into one non-differential graded braided commutative algebra. 
	 This is exactly the description that we need for our study of the differential graded Verlinde algebra in \cite{vd}. 
	\end{example}

Theorem~\ref{thhmcomparisondeligne}
allows for a generalization to exact module categories:
For a finite tensor category $\cat{C}$, an \emph{exact (left) module category $\cat{M}$} \cite{etingofostrik}
is a finite category together with structure of a left module $\otimes : \cat{C}\boxtimes \cat{M}\to\cat{M}$ over $\cat{C}$
 such that $P\otimes M$ is projective for $P\in\Proj\cat{C}$ and $M\in\cat{M}$. 
If we denote by $[-,-]:\cat{M}\boxtimes\cat{M}\to\cat{C}$ the internal hom of the module category $\cat{M}$, one may define the object $\mathbb{A}_\cat{M}:=\int_{M\in\cat{M}} [M,M] \in \cat{C}$. This object is an algebra in $\cat{C}$ and lifts in fact to a braided commutative algebra $\alg_\cat{M}$ in $Z(\cat{C})$ \cite[Theorem~4.9]{shimizucoend}. The object $\mathbb{A}_\cat{M}$ allows to express the Hochschild cochains of $\cat{M}$ as $\rint_{M\in\Proj\cat{M}} \cat{M}(M,M)\simeq \cat{C}(I,\mathbb{A}_\cat{M}^\bullet)$; 
this is \cite[Corollary~7.5]{shimizucoend} in a slightly different language (this is the point where exactness of $\cat{M}$ is needed). After implementing the needed changes to Theorem~\ref{thhmcomparisondeligne} and its proof, we arrive at the following generalization: 
	
\begin{theorem}
	For any exact module category $\cat{M}$
	over a finite tensor category $\cat{C}$,
	the algebra structure on the canonical end 
	$\mathbb{A}_\cat{M}=\int_{M\in\cat{M}} [M,M] \in \cat{C}$ induces  an $E_2$-algebra structure on the 
	derived morphism space $\cat{C}(I,\mathbb{A}^\bullet_\cat{M})$. 
	Under the equivalence $\cat{C}(I,\mathbb{A}^\bullet_\cat{M})\simeq \rint_{M\in \Proj \cat{M}} \cat{M}(M,M)$, this $E_2$-structure
	provides a solution to Deligne's Conjecture in the sense that it induces the standard Gerstenhaber structure on the Hochschild cohomology of $\cat{M}$.
	\end{theorem}

Considering $\cat{C}$ as an exact module category over itself, the above results specializes to Theorem~\ref{thhmcomparisondeligne}.

\spaceplease
\section{The self-extension algebra of a finite tensor category and the Farinati-Solotar Gerstenhaber bracket}
Our  results, in particular Theorem~\ref{thmE2derivedhom} and Theorem~\ref{thhmcomparisondeligne},
allow us to generalize some results in the homological algebra of finite tensor categories (and to simplify the proofs of the existing results).

One of the key homological algebra quantities of a finite tensor category $\cat{C}$ is the \emph{self-extension algebra} $\Ext_\cat{C}^*(I,I)$ of the unit $I$ which was studied, in the framework of finite tensor categories,  by Etingof and Ostrik \cite{etingofostrik}, see \cite{negronplavnik} for an overview.
It is known that $\Ext_\cat{C}^*(I,I)$ is graded commutative. If $\cat{C}$ is the category of finite-dimensional representation of a finite group $G$,
$\Ext_\cat{C}^*(I,I)$ is the group cohomology ring $H^*(G;k)$.
For certain small quantum groups, the Ext algebra was computed by Ginzburg and Kumar \cite{gk}.

If $\cat{C}$ is given by the category of finite-dimensional modules over a finite-dimensional Hopf algebra, Farinati and Solotar
\cite{farinatisolotar}
have given a Gerstenhaber bracket on the self-extension algebra, see also \cite{hermann} for a discussion of the inclusion $\Ext_\cat{C}^*(I,I) \to  HH^*(\cat{C})$ of algebras. 
The appearance of the Farinati-Solotar bracket comes, as we will show below,
 from the fact that under mild conditions the derived endomorphisms $\cat{C}(I,I^\bullet)$ of the unit of a tensor category  actually form an $E_2$-algebra. A similar argument is given, albeit in a non-linear setting,  by Kock and Toën in \cite{kocktoen} in terms of weak 2-monoids and discussed in terms of $B_\infty$-algebras by Lowen and van den Bergh in  \cite{Binfinity}.
Theorem~\ref{thmE2derivedhom} allows us to give a direct proof for the $E_2$-structure on $\cat{C}(I,I^\bullet)$:

\begin{corollary}\label{corcohomftc}
	Let $\cat{C}$ be a finite tensor category. 
	The self-extension algebra
	$\cat{C}(I,I^\bullet)$  carries the structure of an $E_2$-algebra
	that after taking cohomology induces the Farinati-Solotar Gerstenhaber bracket.
	With this $E_2$-structure, there 
	is a canonical map
	\begin{align}\label{eqnExtHochschild}
	\cat{C}(I,I^\bullet)\to \rint_{X\in \Proj\cat{C}}\cat{C}(X,X)  
	\end{align} to the Hochschild cochain complex of $\cat{C}$ equipped with the usual $E_2$-structure. This map is a map of $E_2$-algebras. After taking cohomology, it induces a monomorphism \begin{align}\label{eqnExtHochschild2}
	\Ext_\cat{C}^*(I,I) \to HH^*(\cat{C}) \end{align}
	of Gerstenhaber algebras (with suitable models, it is also a monomorphism at cochain level).
	\end{corollary}

\begin{proof}
	The unit $I$ is trivially an algebra in $\cat{C}$ that also lifts to a braided commutative algebra in $Z(\cat{C})$. This turns $\cat{C}(I,I^\bullet)$ into an $E_2$-algebra by Theorem~\ref{thmE2derivedhom} which by construction is equivalent to the $E_2$-algebra $Z(\cat{C})(\coalg,I^\bullet)$ from Proposition~\ref{propE2onhom}.
	By Theorem~\ref{thhmcomparisondeligne} the Hochschild cochain complex $\rint_{X\in \Proj\cat{C}}\cat{C}(X,X) $ is canonically equivalent as an $E_2$-algebra to $\cat{C}(I,\mathbb{A}^\bullet)$ and also to $Z(\cat{C})(\coalg,\ialg)$.  
	Therefore, up to equivalence, the map~\eqref{eqnExtHochschild} is the map
	\begin{align} Z(\cat{C})(\coalg,I^\bullet)\to Z(\cat{C})(\coalg,\ialg)\label{eqnExtHochschild3} \end{align}
	induced by the unit map $I\to\alg$ of $\alg\in Z(\cat{C})$.  Since this unit map is a morphism of algebras, \eqref{eqnExtHochschild3}  is a map of $E_2$-algebras by Proposition~\ref{propnaturality}. This gives us the  morphism of $E_2$-algebras~\eqref{eqnExtHochschild}.
	
	Next we need to prove that~\eqref{eqnExtHochschild3} is a monomorphism in cohomology and, for suitable models, also at cochain level:
	By means of the adjunction $L\dashv U$, 
	we can identify~\eqref{eqnExtHochschild3} 
	with the map \begin{align}\label{eqnExtHochschild4}\cat{C}(I,I^\bullet)\to\cat{C}(I,\mathbb{A}^\bullet) \ . 
		\end{align}
	Consider the projection $\mathbb{A}\to I\otimes I^\vee \cong I$ to the unit component.
	Then the composition 
	$I\to \mathbb{A}\to I$ is the identity of $I$, 
	i.e.\ $I\to\mathbb{A}$ is a split monomorphism (although $I\to \alg$ in $Z(\cat{C})$ is generally not split).
Since we can model $\mathbb{A}^\bullet$ as $I^\bullet \otimes \mathbb{A}$, we see that $I^\bullet \to \mathbb{A}^\bullet$ is  a split monomorphism as well.
	Hence, the monomorphism $I^\bullet \to \mathbb{A}^\bullet$ is absolute (i.e.\ preserved by any functor). 
	As a consequence, \eqref{eqnExtHochschild4} is a monomorphism and \eqref{eqnExtHochschild2} is also a monomorphism.

	In order to complete the proof, 
	we must compare the structure on cohomology with the one given by Farinati and Solotar
	\cite{farinatisolotar}:
By virtue of \eqref{eqnExtHochschild2} being a monomorphism of Gerstenhaber algebras and Theorem~\ref{thhmcomparisondeligne},
the Gerstenhaber structure that we construct on $\Ext_\cat{C}^*(I,I)$ is a restriction of the usual Gerstenhaber structure on Hochschild cohomology. The same is true for the Farinati-Solotar Gerstenhaber algebra structure by the construction in \cite{farinatisolotar}. Hence, both structures must agree on cohomology.
	\end{proof}

\begin{example}[Quantum groups]
	Let $\catf{u}_q(\mathfrak{sl}_2)$ be again the small quantum group at a primitive root of unity as discussed in \cite{lq}.
	The Ext algebra of $\catf{u}_q(\mathfrak{sl}_2)$ is computed in \cite{gk}; it is supported in even degree. 
	Its Gerstenhaber bracket is zero because $\catf{u}_q(\mathfrak{sl}_2)$ comes with an $R$-matrix (and hence its category of modules with a braiding). This uses that by \cite[Corollary~6.3.17 \& Remark~6.3.19]{hermann}
the Gerstenhaber bracket on $\Ext_\cat{C}^*(I,I)$ vanishes if $\cat{C}$ is braided.
The monomorphism $\Ext_{\catf{u}_q(\mathfrak{sl}_2)}^*(k,k) \to HH^*(  \catf{u}_q(\mathfrak{sl}_2)   )$ is of course only non-trivial in even degree. For the description in even degree, one uses that $ HH^{2*}(  \catf{u}_q(\mathfrak{sl}_2)   )$ is given by $\Ext_{\catf{u}_q(\mathfrak{sl}_2)}^*(k,k) \otimes Z(  \catf{u}_q(\mathfrak{sl}_2)  )$ modulo a certain ideal \cite[Proposition~5.6]{lq}. The map $\Ext_{\catf{u}_q(\mathfrak{sl}_2)}^*(k,k) \to HH^*(  \catf{u}_q(\mathfrak{sl}_2)   )$ is the concatenation
of $\Ext_{\catf{u}_q(\mathfrak{sl}_2)}^*(k,k) \ra{\id \otimes \text{unit}} \Ext_{\catf{u}_q(\mathfrak{sl}_2)}^*(k,k) \otimes Z(  \catf{u}_q(\mathfrak{sl}_2)  )$ with the quotient map. 
The statement of Corollary~\ref{corcohomftc} is that the chain level version of this map is a map of $E_2$-algebras.
\end{example}

\section{Generalizing a result of Menichi\label{secmenichi}}
In this section, we ask the question when the $E_2$-algebras built from homotopy invariants of braided commutative algebras naturally extend
to \emph{framed}
$E_2$-algebras, thereby making their cohomology a Batalin-Vilkovisky algebra.
We will use this to generalize a result of Menichi \cite{menichi}.

First of all, recall from \cite[Theorem~6.1~(4)]{shimizuunimodular} that for any finite tensor category $\cat{C}$ the 
 algebra $\alg=RI \in Z(\cat{C})$
 has the structure a Frobenius algebra if and only if $\cat{C}$ is unimodular.
 We will focus in this section on the case that $\cat{C}$ is additionally 
 pivotal, i.e.\ equipped with a monoidal isomorphism $-^{\vee\vee} \cong \id_\cat{C}$.
 The Frobenius structure on $\alg$ 
 comes from the fact 
 in the case $D\cong I$, we obtain via~\eqref{eqnLandR}
 an isomorphism $\Psi:\alg\to\coalg$
 which is in fact an isomorphism of 
 left $\alg$-modules, 
 and this is one of the many ways to 
 describe a Frobenius structure, see \cite{fuchsstigner}. 
 A direct computation using \cite[Remark~6.2]{shimizuunimodular}
 shows
 	\begin{equation}	  \begin{tikzpicture}
	\begin{pgfonlayer}{nodelayer}
		\node [style=none] (0) at (5, -5.75) {\mlabel{\alg}};
		\node [style=none] (1) at (8, 0.75) {\mlabel{\coalg}};
		\node [style=none] (2) at (6.5, -2) {};
		\node [style=none] (3) at (6.5, -1.25) {};
		\node [style=none] (4) at (5, -1.25) {};
		\node [style=none] (5) at (5, -3.75) {};
		\node [style=none] (6) at (8, -2) {};
		\node [style=none] (7) at (8, 0.25) {};
		\node [style=none] (8) at (5, -5.25) {};
		\node [style=flat box 2] (9) at (7.25, -4) {\mlabel{\Lambda}};
		\node [style=none] (10) at (7.25, -3) {};
	\end{pgfonlayer}
	\begin{pgfonlayer}{edgelayer}
		\draw [bend right=90, looseness=1.50] (3.center) to (4.center);
		\draw (4.center) to (5.center);
		\draw [bend right=90, looseness=2.25] (2.center) to (6.center);
		\draw [style] (5.center) to (8.center);
		\draw [style] (7.center) to (6.center);
		\draw [style] (9) to (10.center);
		\draw [style, in=90, out=-90] (3.center) to (2.center);
	\end{pgfonlayer}
\end{tikzpicture}   \ = \  \Psi    \ ,    \normalsize\label{eqnpsilambda}
 \end{equation}\normalsize
 where $\Lambda :I\to\coalg$ is the so-called \emph{integral} of $\coalg$. 
 In fact, this Frobenius structure is symmetric:

\begin{lemma}\label{lemmasymfrobalg}
	For a unimodular pivotal finite tensor category $\cat{C}$,
	the Frobenius algebra $\alg \in Z(\cat{C})$ is symmetric.
\end{lemma}

\begin{proof}
	We give the proof using	 the language of pivotal module categories from~\cite{relserre}:
	For the canonical algebra $\alg \in Z(\cat{C})$, denote by
	$\alg\mod_{Z(\cat{C})}$ the category of left $\alg$-modules in $Z(\cat{C})$.
	Then $\alg$ can be recovered as the endomorphisms of the left regular
	$\alg$-module $\alg$ in $Z(\cat{C})$;
	in short $\alg = \End_\alg(\alg,\alg)$. By
	\cite[Theorem~6.1~(2)]{shimizuunimodular} $\alg\mod_{Z(\cat{C})}\simeq
	\cat{C}$ as $Z(\cat{C})$-module categories. Since $\cat{C}$ is pivotal,
	$\cat{C}$ is also pivotal as a module category over itself.      Of
	course, $\cat{C}$ is also a module category over $Z(\cat{C})$, and it is in fact a pivotal module category by
	\cite[Corollary~38]{internal}. Therefore, $\alg\mod_{Z(\cat{C})}$ is
	also a pivotal $Z(\cat{C})$-module category.    By $\alg =
	\End_\alg(\alg,\alg)$ the object $\alg\in Z(\cat{C})$ can be recovered
	as the endomorphism object of an object  in a pivotal module category and hence inherits the
	structure of a symmetric Frobenius algebra in $Z(\cat{C})$ by
	\cite[Theorem~3.15]{relserre}.
\end{proof}

\begin{lemma}\label{lemmaE2derivedhompiv}
	For any unimodular pivotal finite tensor category $\cat{C}$, the canonical algebra $\alg \in Z(\cat{C})$ is not only braided commutative, but framed braided commutative in the sense that additionally the balancing of $\alg$ is trivial, $\theta_\alg =\id_\alg$.
	The same is true for the canonical coalgebra $\coalg \in Z(\cat{C})$. 
\end{lemma}

\begin{proof} The canonical algebra $\alg \in
	Z(\cat{C})$ is a symmetric Frobenius algebra by Lemma~\ref{lemmasymfrobalg}
	and 
	braided commutative (Lemma~\ref{lemmadmno}) (note that the braided commutativity does
	\emph{not} imply the symmetry because we are not working in a symmetric
	category).           Thanks to
	\cite[Proposition~2.25~(i)]{correspondences}, this implies that the
	balancing of $\alg$ is trivial, $\theta_\alg =\id_\alg$. The same
	holds true for the canonical coalgebra $\coalg\in Z(\cat{C})$ because $\alg \cong \coalg$ as objects in $Z(\cat{C})$ thanks to unimodularity. 
\end{proof}

The Drinfeld center $Z(\cat{C})$ of a pivotal finite tensor category $\cat{C}$ comes with an induced pivotal structure.
As usual, a pivotal structure allow us to define a balancing on $X\in Z(\cat{C})$ by
\begin{align}\theta_X= \begin{tikzpicture}
	\begin{pgfonlayer}{nodelayer}
		\node [style=none] (20) at (-14, 19.5) {\mlabel{X}};
		\node [style=none] (83) at (-14, 24) {\mlabel{X}};
		\node [style=none] (98) at (-11.75, 21.25) {\mlabel{X^\vee}};
		\node [style=none] (100) at (-13, 22.5) {};
		\node [style=black dot] (101) at (-12, 22) {};
		\node [style=none] (102) at (-13, 21.5) {};
		\node [style=none] (103) at (-12, 22) {};
		\node [style=none] (104) at (-14, 23.5) {};
		\node [style=none] (105) at (-14, 20) {};
		\node [style=none] (106) at (-11.75, 22.75) {\mlabel{{^\vee X}}};
	\end{pgfonlayer}
	\begin{pgfonlayer}{edgelayer}
		\draw [in=-90, out=-60, looseness=1.25] (102.center) to (103.center);
		\draw (102.center) to (104.center);
		\draw [style=over] (105.center) to (100.center);
		\draw [in=90, out=60, looseness=1.75] (100.center) to (101);
		\draw (103.center) to (101);
	\end{pgfonlayer}
\end{tikzpicture} \ ,  \label{eqnbalancingdef}
\end{align} 
where the black dot is the natural isomorphism $X^\vee \cong {^\vee X}$ 
given by the pivotal structure. 

We are now ready to prove a statement on framed extensions of the $E_2$-algebras constructed 
 from homotopy invariants of braided commutative algebras:
 
\begin{theorem}\label{thmE2derivedhompiv}
	Let $\mathbb{T} \in \cat{C}$ be an algebra in a unimodular pivotal finite tensor category $\cat{C}$
	with a lift to a framed braided commutative algebra $\somealg \in Z(\cat{C})$.
	Then the multiplication  and the half braiding of $\mathbb{T}$ induce the structure of a framed
	$E_2$-algebra on the space $\cat{C}(I,\mathbb{T}^\bullet)\simeq Z(\cat{C}) ( \coalg, \isomealg      )$ of homotopy invariants of $\mathbb{T}$. 
\end{theorem}

\begin{proof}If $\cat{C}$ is pivotal, then  $Z(\cat{C})$ is balanced via~\eqref{eqnbalancingdef}; operadically speaking, this is an extension   to a \emph{framed $E_2$-algebra} in categories~\cite{salvatorewahl}.
	Passing from $E_2$ to framed $E_2$ means passing from braid groups to framed braid groups. Therefore,
	it is straightforward to observe that Proposition~\ref{propE2onhom} remains true if we replace the braided category by a balanced braided category and the braided commutative algebra by a framed braided commutative algebra.
	We can now proceed as in the proof of Theorem~\ref{thmE2derivedhom} because $\coalg$ is \emph{framed} braided cocommutative by Lemma~\ref{lemmaE2derivedhompiv}.
\end{proof}

Theorem~\ref{thmE2derivedhompiv} has the following application:
Menichi proves in \cite[Theorem~63]{menichi}
that for a finite-dimensional pivotal and unimodular Hopf algebra $A$, the inclusion $\Ext_A^*(k,k)\to HH^*(A;A)$ is not only a monomorphism of Gerstenhaber algebras, but actually a monomorphism of Batalin-Vilkovisky algebras. 
We can use Theorem~\ref{thmE2derivedhom} to give a generalization of this result to a result at cochain level that holds for all unimodular pivotal finite tensor categories, not only those coming from pivotal and unimodular Hopf algebras.

\begin{corollary}\label{corgenmen}
	For any unimodular pivotal finite tensor category $\cat{C}$, both the self-extension algebra $\cat{C}(I,I^\bullet)$ and the Hochschild cochain complex $\rint_{X\in\Proj\cat{C}}\cat{C}(X,X)$ come equipped with  a framed $E_2$-algebra structure such that 
	\begin{align}
	\cat{C}(I,I^\bullet)\to \rint_{X\in \Proj\cat{C}}\cat{C}(X,X)   
	\end{align}  is a map (and with suitable models even a monomorphism) of framed $E_2$-algebras. After taking cohomology, it induces a monomorphism
	\begin{align}
	\Ext_\cat{C}^*(I,I) \to HH^*(\cat{C}) 
	\end{align} of Batalin-Vilkovisky algebras.
\end{corollary}

\begin{proof}
	We obtain the \emph{framed} $E_2$-algebras $\cat{C}(I,I^\bullet)$ and $\rint_{X\in \Proj\cat{C}}\cat{C}(X,X)\simeq \cat{C}(I,\mathbb{A}^\bullet)$ thanks to
	Theorem~\ref{thmE2derivedhompiv} (we need to use Lemma~\ref{lemmaE2derivedhompiv} again).
	Then we can proceed as for Corollary~\ref{corcohomftc}.
	\end{proof}

\small 
\spaceplease

\end{document}